\documentclass[11pt,a4paper]{amsart}
\usepackage{amssymb}
\usepackage[dvips]{graphicx}
\usepackage{psfrag}
\usepackage{amscd}
\usepackage[latin1]{inputenc}

\numberwithin{equation}{section}

\newcommand{\comm}[1]{}
\newcommand{\ud}{\,\mathrm{d}}

\def\diam{\operatorname{diam}}
\def\dist{\operatorname{dist}}

\def\ti{\tilde}

\def\({\left(}
\def\){\right)}
\def\oli{\overline}
\def\uli{\underline}

\def\raw{\rightarrow}

\def\no={\neq}
\def\sm{\setminus}

\def\C{{\mathbb C}}

\def\P{{\mathbb P}}

\def\BB{{\mathcal B}}

\def\EE{{\mathcal E}}

\def\OO{{\mathcal O}}

\def\RR{{\mathcal R}}

\def\al{\alpha}
\def\be{\beta}
\def\ga{\gamma}
\def\de{\delta}

\def\vep{\varepsilon}

\def\th{\theta}

\def\ka{\kappa}
\def\la{\lambda}

\def\si{\sigma}

\def\om{\omega}

\def\Ga{\Gamma}
\def\De{\Delta}

\def\La{\Lambda}

\def\Om{\Omega}

\theoremstyle{plain}
\newtheorem{Main}{Theorem}

\newtheorem{Thm}{Theorem}[section]
\newtheorem{Prop}[Thm]{Proposition}
\newtheorem{Lem}[Thm]{Lemma}

\theoremstyle{remark}
\newtheorem{Rem}[Thm]{Remark}
\newtheorem{Def}[Thm]{Definition}

\begin{document}

\title{Slowly recurrent Collet-Eckmann maps on the Riemann sphere}

\author{Magnus Aspenberg}
\address{Department of Mathematics, Lund University, Box 118, 221 00 Lund, Sweden}

\maketitle

\begin{abstract}
  In this paper we study perturbations of rational Collet-Eckmann maps for which the Julia set is the whole sphere, and for which the critical set is allowed to be slowly recurrent. Generically, if each critical point is simple, we show that each such Collet-Eckmann map is a Lebesgue point of Collet-Eckmann maps in the space of rational maps of the same degree $d \geq 2$. The same result holds in each subspace, where we fix the multiplicities of the critical points.
\end{abstract}

\section{Introduction}

The Collet-Eckmann condition stems from J-P. Eckmann and P. Collet in the 1980s \cite{Collet-Eckmann1, Collet-Eckmann2}, and was used to show abundance of chaotic behaviour for certain maps on an interval. Chaotic behaviour of a system is usually associated to the property of sensitive dependence on initial conditions, meaning that two points $x,y$ sufficiently close to each other repel each other under iteration up to some large scale. Hence it is natural that such maps possess some kind of expanding property. A map satisfying the Collet-Eckmann condition is expansive along the forward critical orbit(s), and it turned out to be sufficient for chaotic behaviour in many situations, not only the pioneering case studied by J-P. Eckmann and P. Collet. Shortly after their works, M. Jakobson proved in \cite{MJ} that the set of parameters $a \in (0,2)$ for which $f_a(x) = 1-ax^2$ admits an invariant absolutely continuous measure (acim) has positive Lebesgue measure. A corresponding celebrated result for complex rational maps was obtained by M. Rees in \cite{MR}. These maps also exhibit chaotic behaviour. The existence of an acim describes the typical orbits of a map in a probabilistic way. It does not immediately imply chaotic behaviour, but it is often very closely related to it and with some additional properties (such as expansion, ergodicity, positive entropy etc) this is usually the case.

It was quite early realised that the Collet-Eckmann condition, or even weaker conditions, are sufficient for the existence of an (ergodic) acim, see e.g. \cite{Collet-Eckmann3}, \cite{BC1, BC2}, \cite{Keller}, \cite{Nowicki-Strien}, \cite{BSS}, \cite{Misse-IHES}, \cite{Benedicks-Misse}, \cite{GS2}, \cite{FP2}. In the fundamental papers \cite{BC1,BC2}, M. Benedicks and L. Carleson showed that the Collet-Eckmann condition is satisfied for a set of positive Lebesgue measure in the quadratic family. Despite of the fact that the Collet-Eckmann condition in general is stronger than the existence of an acim, the two conditions are metrically the same in the real quadratic family. This was a deep result by M. Lyubich and A. Avila and C. G. Moreira, see \cite{AM}, \cite{ML}. Conjecturally it holds more generally. In contrast to the chaotic, non-regular (sometimes called stochastic) parameters stands the regular parameters, for which the map has an attracting orbit. These maps were proven to be open and dense in the real case (the famous real Fatou conjecture), \cite{ML2}, \cite{GSW}, \cite{SSK}. The complex Fatou conjecture is still open.

In the complex rational setting, not as much is known. A similar result to the papers \cite{BC1, BC2} was obtained by the author \cite{MA-Z}, improving the result by M. Rees \cite{MR}. Apart from implying the existence of an ergodic acim, the Collet-Eckmann condition induces more nice properties, see e.g. \cite{PRS}, \cite{GS2}. It has geometric implications, and there are several papers studying perturbations of Collet-Eckmann maps (or similar expanding maps); see \cite{FP3}, \cite{GS}, \cite{BC1}, \cite{BC2}, \cite{Tsuji}, \cite{Gao-Shen} for real maps on an interval and families of H\'enon maps, and \cite{MR}, \cite{MA-Z}, \cite{Mih-Gaut-etal},  \cite{GSW-Fine}, \cite{GSW-Struc} in the complex setting. 

The result in this paper is related to \cite{Tsuji} (see also \cite{Gao-Shen}) in the complex setting. 
We study perturbations of complex rational Collet-Eckmann maps which have their Julia set equal to the whole sphere, and where the starting map is allowed to be critically slowly {\em recurrent} (see \cite{Tsuji} and \cite{Le-Pr-Sh}).

Let $Crit$ be the set of critical points for $f$ and let $J(f)$ and $F(f)$ be the Julia set and Fatou set of $f$ respectively. Let $Crit'$ be the set of critical points $c$ such that there are no other critical points in the forward orbit of $c$. Derivatives are always in the spherical metric unless otherwise stated. We write $Df(z) = f'(z)$ as the space derivative throughout the paper. 
\begin{Def}
Let $f$ be a non-hyperbolic rational map without parabolic periodic points. Then $f$ satisfies the {\em Collet-Eckmann condition} (CE), if there exist constants $C > 0$ and $\ga > 0$ such that, for each critical point $c \in Crit' \cap J(f)$, we have
\[
|Df^n(fc)| \geq Ce^{\ga n}, \text{  for all $n \geq 0$}.
\]
\end{Def}


Let us define the upper and lower Lyapunov exponents for the critical point $c$ respectively as
\begin{equation}
\uli{\ga}(c) = \liminf\limits_{n \raw \infty} \frac{\log |Df^n(fc)|}{ n}, \quad \text{and} \quad
\oli{\ga}(c) = \limsup\limits_{n \raw \infty} \frac{\log |Df^n(fc)|}{ n}. \nonumber
\end{equation}
Then the CE-condition can be reformulated as the condition that the lower Lyapunov exponent is strictly positive for all critical points $c \in Crit' \cap J(f)$.  We write $\uli{\ga} = \min \uli{\ga}(c)$ where the minimum is taken over all critical points $c \in Crit' \cap J(f)$. 
In this paper we are going to study perturbations of rational CE maps for which the Julia set is the whole sphere, but we expect that the techniques can be used in other situations as well. A critical point $c \in J(f)$ is {\em slowly recurrent}, cf. \cite{Le-Pr-Sh}, if for each $\al > 0$ there is some $C > 0$ such that
\begin{equation} \label{slowrecur}
\text{dist}(f^n(Jrit),Jrit) \geq Ce^{-\alpha n}, \text{ for all $n \geq 0$}.
\end{equation}
We say that $f$ is critically slowly recurrent if all critical points in the Julia set are slowly recurrent. 
Collet-Eckmann maps possess a (unique) conformal measure $\nu$ supported on the Julia set and a unique ergodic invariant measure $\mu$, which is absolutely continuous w.r.t. $\nu$ (e.g. \cite{FP2}, \cite{PRS}, \cite{GS}). If the map $f$ satisfies $J(f) = \hat{\C}$, then $\nu$ is the standard Lebesgue measure and hence for such maps there exists an invariant absolutely continuous measure w.r.t. Lebesgue measure.  We say that the critical points are {\em typical} with respect to this measure if the Birkhoff means converges for all critical points $c \in Jrit$, i.e.,
\[
\frac{1}{n} \sum_{k=0}^{n-1} \varphi(f^k(c)) \raw \int \varphi  \ud \mu,  \quad \text{ as $n \raw \infty$},
\]
for $\varphi \in L^1(\mu)$. Setting $\varphi(z) = \log |Df(z)|$, which belongs to $L^1 (\mu)$ by \cite{PRS}, we see that if the critical points are typical, then $\oli{\ga} = \uli{\ga}$. It follows that the map is slowly recurrent. The condition $\oli{\ga} = \uli{\ga}$ implies that $f$ is slowly recurrent but it is not clear if the converse holds. Conjecturally almost all CE-maps have the slow recurrence property. At least it is true in the real quadratic family (see \cite{AM}).

The space of rational maps of degree $d$ is denoted by $\RR_d$. We will explain later what is meant by a non-degenerate real analytic family in Section \ref{L-parameter-space}. 

\begin{Main} \label{sats1}
Let $f$ be a critically slowly recurrent rational Collet-Eckmann map in $\RR_d$, of degree $d \geq 2$, such that the Julia set is the whole sphere and let $f_a$, $a \in (-\vep,\vep)$ be a non-degenerate real analytic family of maps around $f=f_0$ for some $\vep > 0$. Then $f_0$ is a Lebesgue density point of Collet-Eckmann maps in $(-\vep,\vep)$. 
\end{Main}

We use a normalisation of the space of rational maps following Levin \cite{Levin-book, Levin-multipliers}. We say that two maps $f$ and $g$ are equivalent if they are conjugate by a M\"obius transformation. Then we can consider the space $\La_{d,\oli{p'}} \subset \RR_d$, (see \cite{Levin-book}) up to equivalence, as the set  of rational maps $f$ of degree $d \geq 2$ with precisely $p'$ critical points, i.e. $Crit = \{ c_1, \ldots, c_{p'} \}$, with corresponding multiplicities $\oli{p'} = \{m_1, \ldots, m_{p'} \}$ (in the same order). 
Inside such a set, we may state the following theorem as a direct consequence of Theorem \ref{sats1}, by Fubini's Theorem, since the set of non-degenerate directions has measure zero (this also follows from the results of G. Levin, see Section \ref{L-parameter-space}). 
\begin{Main} \label{sats2}
Let $f$ be a critically slowly recurrent rational Collet-Eckmann map in $\La_{d,\oli{p'}} \subset \RR_d$ of degree $d \geq 2$ such that the Julia set is the whole sphere. Then $f$ is a Lebesgue density point of Collet-Eckmann maps in $\La_{d,\oli{p'}}$. 
\end{Main}
Generically, all critical points are simple and then $f$ is a Lebesgue density point of CE-maps in $\RR_d$ in Theorem \ref{sats2}. 
It is likely that this is true even if the critical points are not simple, using a suitable reparameterisation of the parameter space so that all critical points move analytically if higher order critical points split. However, we will not consider that case in this paper.  
  

The proof of the main theorem is mainly based on a combination of strong results on transversality by G. Levin and developed classical Benedicks-Carleson parameter exclusion techniques. In particular, it is a generalisation of \cite{MA-Z}, which was the (revised) thesis of the author. 
Apart from proving Theorem \ref{sats1}, the aim of this paper is partially to make the arguments in the Benedicks-Carleson parameter exclusion techniques more transparent. 

\begin{Rem}
It will be clear from the proof that the slow recurrence condition in Theorems \ref{sats1} and \ref{sats2} is a little superfluous; one only needs to have slow recurrence (\ref{slowrecur}) for {\em some} sufficiently small $\al > 0$, depending on $f=f_0$. The CE-maps constructed in \cite{MA-Z} have this property close to the starting (Misiurewicz-Thurston) map. It follows that the set of maps satisfying this weaker assumption has positive Lebesgue measure. 
\end{Rem}

\newpage

\subsection*{Acknowledgements}
The author wants to thank G. Levin for many fruitful conversations on transversality. The author thanks M. Bylund and W. Cui for many interesting discussions, remarks and corrections, and finally expresses his gratitude to the Department of Mathematics at Lund University.

\section{Some definitions}

Let $f=f_0$ be a slowly recurrent Collet-Eckmann map, and $f_a$, $a \in (-\vep,\vep)$ a real analytic family around $f_0$. We assume that the family is {\em non-degenerate}, which in particular means that every critical point $c_l(a)$ of $f_a$ moves analytically with the parameter $a$.
Another condition on the family is given in Section \ref{L-parameter-space}. We put $f_a(z) = f(z,a)$ and $Df_a(z) = f_a'(z)$. 
Let $v_l(a) = f_a(c_l(a))$ be the critical value, and suppose that $v_l = v_l(0)$ does not contain any critical points in its forward orbit under $f_0$, for all $l$. Put
\[
\xi_{n,l}(a) = f_a^n(c_l(a)).
\]

We will study the evolution of $\xi_{n,l}(\om)$ for a small interval $\om = (-\vep,\vep)$ around the starting map $f_0$. In the beginning this curve will grow rapidly from the expansive properties of the starting map, but later on we have to delete parameters that come too close to the set of critical points, denoted by $Crit_a$, of $f_a$. Now, $Crit_a$ moves analytically, but it turns out that $\xi_{n,l}(\om)$ and $Crit_{\om}$ are very different in diameter, due to the expansion of $\xi_{n,l}(\om)$; it will be much bigger than $\diam(Crit_{\om})$.  Let $U$ be a neighbourhood of the critical points for the unperturbed map. Choose $\vep > 0$ so that $U$ is a neighbourhood around $Crit_a$, for all $a \in (-\vep,\vep)$. Moreover, if we let $U_l$ be a component of $U$ which contains the critical point $c_l$ then we impose the condition $\dist(c_l(\om),\partial U_l) \gg \diam( c_l(\om) )$ for all $l$. To make $U$ more precise, we choose $\de  = e^{-\De} > 0$ so that $U = \cup_l B(c_l,\de)$. Hence $\vep$ depends on $\de$. 

The approach rate at which the distance $\dist(\xi_{n,l}(a),Crit_a)$ may go to zero is controlled by the so called {\em basic approach rate assumption} which is inherited from the slow recurrent condition.

\begin{Def}
Let $\al > 0$. We say that the critical point $c_l(a)$, (or parameter $a$ with critical point $l$) satisfies the {\em basic assumption up to time $n$ with exponent $\al$,} if
\[
\dist(\xi_{n,l}(a),Crit_a) \geq K_b  e^{-2\al k}, \quad \text{ for all $k \leq n$},
\]
where $K_b > 0$ is the same constant which appears in the slow recurrent condition.
\end{Def}

Obviously the starting map $f_0$ satisfies the basic assumption for all times for any $\al > 0$. From now on, fix $\al > 0$ to be much smaller than $\min(\ga_0, \ga_H) (1-\tau) /(400K \Ga)$, where $\ga_0 = \oli{\ga}=\uli{\ga}$ is the Lyapunov exponent appearing in the Collet-Eckmann condition for $f_0$, $\Ga = \sup_{a \in (-\vep,\vep), z \in \hat{\C}} \log |Df_a(z)|$, $K$ is the maximal degree of the critical points, $\ga_H$ is the exponent from Lemma \ref{oel}, and where $0 < \tau < 1$ is fixed (this is used in Section \ref{large-deviations}). We assume for the starting map $f_0$, that there is some constant $C_0 > 0$ such that,
\[
|Df^n(v_l(0))| \geq C_0 e^{\ga_0 n}, \quad \text{ for all $n \geq 0$.}
  \]
We will construct a set of parameters around $a=0$ which also satisfies both this basic assumption for this specific $\al$ and the Collet-Eckmann condition for possibly slightly smaller Lyapunov exponents $\ga$. Since we fix $\al > 0$ we only speak of the basic assumption, without mentioning the exponent in the future.  

We will make an induction argument based on the fact that we have some ``basic'' Lyapunov exponent $\ga_B > 0$. This is typically smaller than the original Lyapunov exponent $\ga_0$ for $f_0$. During the induction arguments we also have to allow the Lyapunov exponent to decrease down to some certain value, a fraction of $\ga_B$, a so called ``intermediate exponent'' $\ga_I < \ga_B$, which is required for most lemmas to work.  
We will also define the number $\ga_B$ later, but it is slightly smaller than the minimum of $\ga_H$ in Lemma \ref{oel} and $\ga_0$ from the starting function $f_0$. 

We write $A \sim_{\ka} B $, where $\ka \geq 1$ if 
\[
\frac{1}{\ka} A \leq B \leq \ka B. 
\]
We write $A \sim B$ to say $A \sim_{\ka} B$ for some constant $\ka \geq 1$. In several inequalities we use $C$ several times for possibly different constants, when it is clear that these constants do not depend on the dynamics, i.e. the number of iterations.

\subsection{Bound and free periods}

In this section we define some fundamental concepts which will be used throughout the paper. Many of them are direct analogues of corresponding definitions in \cite{BC1,BC2}, see also \cite{MA-Z}. We speak of a {\em return} of the sequence $\xi_{n,l}(\om)$ into $U$ or $U'$, when we mean that $\xi_{n,l}(\om) \cap U \neq \emptyset$ or $\xi_{n,l}(\om) \cap U' \neq \emptyset$ respectively. We also speak of returns into $U$ or $U'$ of the sequence $\xi_{n,l}(a)$ for a single parameter $a$, and this means simply that $\xi_{n,l}(a) \in U$ or $\xi_{n,l}(a) \in U'$ respectively. Returns into the annular neighbourhoods $U' \sm U$, i.e. when $\xi_{n,l}(a) \cap U' \neq \emptyset$ but $\xi_{n,l}(\om) \cap U = \emptyset$, are called {\em pseudo-returns}. Sometimes we drop the index $l$ and write only $\xi_n(a) = \xi_{n,l}(a)$ for some critical point $c_l(a)$. We will also consider so called {\em deep returns}, which are returns into a smaller neighbourhood $U^2 = \cup_l B(c_l,\de^2) \subset U$ of the critical points. These deep returns will be used only in the end of the paper, in Section \ref{large-deviations}. 

The point is that when a return occurs, so that for example $\xi_n(a) \in U$, then the orbit follows the original orbit, i.e. $\xi_{n+j}(a)$ stays close to $\xi_j(a)$ for the first $j$. This is the so called {\em bound period}, which can be defined both for points $\xi_n(a)$ and curves $\xi_n(\om)$ (precise definitions below). After the bound period ends, the {\em free period} starts until the next return, and so on.  During the bound period, due to the expansion of the derivative of the original early orbit, we can show expansion of the derivative also during the bound period (with a certain loss due to the actual return, which is close to the critical set). Because of this, we will not consider  returns during the bound period (bound returns), but only consider returns after the free period (free returns). When we speak of a return, we mean a free return unless otherwise stated. This is very similar to earlier constructions in \cite{MA-Z} and \cite{BC1,BC2}. During the free period we will show a uniform expansion of the derivative. The result is the same as in the old traditions but the techniques stem from quite different sources in this new situation of a more general CE-map. The number $\be > 0$ below is related to $\al$ in the basic approach rate condition. There is lots of freedom, but let us set $\be = \al$, so that we can use the same exponent. 

\begin{Def}[Pointwise bound period]
Let $\be > 0$. Let $\xi_{n,l}(a) \in U_k' \subset U'$ be a return. Then we define the {\em bound period} for this return as the indices $j > 0$ for which the inequality
\[
|\xi_{n+j,l}(a) - \xi_{j,k}(a)| \leq e^{-\be j} \dist(\xi_{j,k}(a), Crit_a), 
\]
holds. The largest number $p > 0$ for which the inequality holds is called the length of the bound period.  \\
\end{Def}

To define the bound period for an interval, we consider a return $\xi_{n,l}(\om)$ into $U$. If 
\begin{equation} \label{essential-return}
  \diam(\xi_{n,l}(\om)) \geq \frac{1}{2} \dist(\xi_{n,l}(\om),Crit_{\om})/ (\log (\dist(\xi_{n,l}(\om),Crit_{\om})))^2,
\end{equation}
then we say that the return is {\em essential}. Otherwise it is {\em inessential}. 
With $\ti{r} = -\log (\dist(\xi_{n,l}(\om),Crit_{\om}))$, then the return is essential if  $\diam(\xi_{n,l}(\om)) \geq (1/2) e^{-\ti{r}}/\ti{r}^2$, a bit more convenient notation. Actually we will partition the parameter intervals (explained later) so that they become so called partition elements, defined as follows. 

\begin{Def} \label{partition-element}
For a given $S > 0$, we call parameter intervals $\om$ satisfying the inequality
\[
\diam(\xi_{k,l}(\om)) \leq \left\{
\begin{array}{cc}
\frac{\dist(\xi_{k,l}(\om),Jrit_{\om})}{(\log (\dist(\xi_{k,l}(\om),Jrit_{\om})))^2} , & \text{if } \xi_{k,l}(\om) \cap U \neq \emptyset, \nonumber  \\
S, & \text{if } \xi_{k,l}(\om) \cap U = \emptyset, \nonumber
\end{array} \right.
\]
for all $k \leq n$, {\em partition elements} at time $n$. 
\end{Def}

We do not speak of essential or inessential returns for pseudo-returns. 

\begin{Def}[Bound period for an interval, essential returns or pseudo-returns]
Let $\xi_{n,l}(\om) \cap U_k' \neq \emptyset$, ($U_k' \subset U'$) be an essential return or a pseudo-return. Then we define the bound period for this return as the indices $j > 0$ for which the inequality (recall $f_a(z) = f(z,a)$),
\[
\dist(f^j(z,a), \xi_{j,k}(b))  \leq e^{-\be j} \dist(\xi_{j,k}(b), Crit_b),
\]
holds for all $a, b \in \om$, and all $z \in  \xi_{n,l}(\om) $. 
\end{Def}

If the return $\xi_{n,l}(\om)$ into $U_k$ is inessential we will consider a host-curve as follows. Draw a straight line segment $L'$ through the end points of $\xi_{n,l}(\om)$ with length equal to $e^{-r}/r^2$ where  
\[
r = \lceil  -\log (\dist(\xi_{n,l}(\om),Crit_{\om}))   - 1/2  \rceil.
\]
To make it well defined, let us say that the line segment $L'$ shall be symmetric with respect to the end points of $\om$. Let $L$ be the part of $L'$ with the central part between the end points deleted. The {\em host curve} for this return is then $L \cup \xi_{n,l}(\om)$. 

\begin{Def}[Bound period for an interval, inessential returns]
Let $\xi_{n,l}(\om) \cap U_k \neq \emptyset$ be an inessential return. Then we define the bound period for this return as the indices $j > 0$ for which the inequality
\[
\dist(f^j(z,a), \xi_{j,k}(b))  \leq e^{-\be j} \dist(\xi_{j,k}(b), Crit_b),
\]
holds for all $a, b \in \om$, and all $z \in L \cup \xi_{n,l}(\om)$. 
\end{Def}

It will be clear later that the dependence on the parameter in these definitions is inessential.

\section{Expansion during the free period.}

During the free period we want to show that the derivative of $f^n(z)$ grows exponentially as long as $f^j(z)$ stays outside $U$ for $j=0,\ldots,n-1$. In earlier papers, this was settled via the orbifold metric for postcritically finite (rational) maps, given that the postcritical set consists of at least $3$ points. Here we have to use different techniques to build a uniform expansion using the second Collet-Eckmann condition discussed in \cite{GS}. In Proposition 1 of that paper, it is stated that the second Collet-Eckmann condition is satisfied for all critical points of maximal multiplicity. However, with the slow recurrence condition, this statement holds for every critical point in the Julia set \cite{Mats-2}. 

Without going through the whole construction, we refer to \cite{GS} and \cite{GS2} for the details. The main idea is based on three types of iterated preimages of shrinking neighbourhoods of a given point $z$, which in our case is a critical point $c$ in the Julia set (actually we assume that $J(f) = \hat{\C}$). This critical point is assumed not to have any critical points in its backward orbit. The type 2 and type 3 orbit have a uniform expansion automatically by construction, see Lemma 3 and Lemma 4 in \cite{GS}. The type 1 preimages connects two critical points in the backward orbit in a way that one has a ball $B(c,r)$ and considers preimages $U_k$ which are sequences of components of $f^{-k}(B(c,r_k)$ of shrinking neighbourhoods $B(c,r_k)$, where $r_k \leq r$ is decreasing and $\lim\limits_{k \raw \infty} r_k \geq r/2$. For a type 1 orbit one has a critical point $c_1 \in \partial U_n$ for some $n$, and no critical points in $\oli{U_k}$ for $0 < k < n$. The length of this type 1 orbit is $n$. Due to the difference in multiplicity of the critical points, type 1 orbits do not ensure immediate uniform expansion. This is resolved by looking at preimages of the type $\ldots 111113$, i.e. a sequence of $1$s followed by a type 3 orbit. Such iterated preimages have uniform expansion (see p. 83 in \cite{GS}).

What can happen is that the induction starts (from the right) with a sequence of $1$s only. Then it may happen that we do not have the desired expansion. Looking at such a block of $1$s in the beginning of the sequence, we see from the calculations on p. 83 \cite{GS} that at a preimage $y  = f^{-k}(c)$ we can estimate the growth of the derivative as follows. Let $\mu_{max}$ be the maximal multiplicity of the critical points and $\mu$ the multiplicity of $c$. The number $d=0$ below because that is the distance from the centre of the ball $B(c,r)$ to the critical point $c$. For some $Q > 1$ we then have, verbatim,  
\begin{equation} \label{backward-slow}
|Df^k(y)|^{\mu_{max}} \geq Q^k \frac{r^{\mu_{max}-1}}{(r+d)^{\mu-1}} = Q^k r^{\mu_{max}-\mu}.
\end{equation}
So if $\mu < \mu_{max}$ then this expansion is not uniform. By assumption there is a critical point $c_1$ on the boundary of the shrinking neighbourhood of $f^{-k}(B(c,r))$, for some $k$, i.e. $c_1 \in \partial U$, where $U = \text{comp}( f^{-k}(B(c,r_k))$ where $r/2 \leq r_k \leq r$. However, by the slow recurrence condition, we have $\dist(f^k(c_1),c) \geq e^{-\al k}$ for some small $\al > 0$. This means that $e^{-k \al} \leq r$. Since $\al > 0$ can be chosen as small as we like, (\ref{backward-slow}) becomes, for some $Q_1 > 1$ possibly slightly smaller than $Q$,
\[
|Df^k(y)|^{\mu_{max}} \geq Q^k e^{-(\mu_{max} -\mu) \al k} = Q_1^k. 
\]

\begin{Lem} \label{oel}
There exist a neighbourhood $U'$ of the critical points such that the following holds. Let $U \subset U'$. There exist $\la > 1$ and $C' > 0$, where $C'$ depends on $\de'$ but not on $\de$ such that; if $f_a^k(z) \notin U$ for $k=0,\ldots, n-1$, then 
\[
|Df_a^n (z)| \geq C e^{-K\De} \la^n.
\]

For each $0 < q \leq 1$ there exists a neighbourhood of the critical points $\hat{U} \subset U'$ such that for any neighbourhood of the critical points $U_1 \subset U \subset \hat{U}$ satisfying $\diam(U_{1,j}) \geq q \diam(U_j)$, where $U_{1,j} \subset U_j$ are components of $U_1$ and $U$ respectively, we have the following. If $z \notin U_1$, $f_a^k(z) \notin U$ for $k=1,\ldots, n-1$, and $f_a^n(z) \in U$ then
\[
|Df_a^n (z)| \geq C' \la^n,
\]
(were $C'$ only depends on $U'$). If $q=1$ we can set $U_1=U$ and $\hat{U} = U'$. 
\end{Lem}
\begin{proof}
Let us first consider the unperturbed map $f_0$. By the argument before the lemma, it follows from \cite{GS} that the Collet-Eckmann condition implies the second Collet-Eckmann condition, for all critical points. Looking at any iterated preimage $z=f^{-n}(c)$ to a critical point $c$, the second Collet-Eckmann condition implies 
  \[
|Df^n(z)| \geq C_2 \la_2^n , 
\]
for some $\la_2 > 1$ and a constant $C_2 > 0$.  Let $0 < \ka < 1$ and $N > 0$ (we give conditions ot these constants below). We follow partially the idea of \cite{PRS} (pp. 40--41).  Let $U'$ to be a union of disks $U_j'$  around the critical points with radius $\de'$, so that for any iterated preimage $f^{-k}(U_j')$ of a component of $U'$, we have 
\begin{equation} \label{dyadic}
\diam (f^{-k}(U_j')) \leq \ka \cdot \dist(f^{-k}(U_j'),Crit_0), \quad \text{ for all $k \leq N$}. 
\end{equation}
This implies that we have distortion inside $f^{-k}(U_j')$, that is, for any choice of $z,w$ in the same component of $f^{-k}(U_j')$ we have
\begin{equation} \label{distortion-est}
  \frac{|Df(z)|}{|Df(w)|} \leq C_3,
\end{equation}
where $C_3= C_3(\ka) \raw 1$, as $\ka \raw 0$.

If (\ref{dyadic}) is not valid, then we can use another estimate as follows.
For any disk $D$ of radius at most $\de' > 0$ there is a constant $C_4$ such that
\[
|Df(z)| \diam(D') \leq C_4 \diam (D), \text{ for all $z \in D'$,}
\]
where $D'$ is a component of $f^{-1}(D)$. Here $C_4$ only depends on $\de'$. 

Suppose now that $N_0 > 0$ is the largest time where (\ref{dyadic}) is valid. If we put $W_k' = f^{-k}(U_j')$ and $z_k \in W_k'$ the corresponding preimage of $c_j \in U_j'$, and if $N$ is large enough, then
  \begin{align}  \nonumber
    \diam(U_j') \geq C_3^{-(N_0-1)}|Df^{N_0}(z_{N_0-1})|   \diam (W_{N_0-1}')  \nonumber \\
   \geq C_4^{-1} C_3^{-(N_0-1)} |Df^{N_0}(z_{N_0})| \diam (W_{N_0}') &\geq \la_1^{N_0} \diam(W_{N_0}'),
\end{align}
for some $\la_1 > 0$. Now let $N$ be so large so that this holds and also that $\la_1^N \geq 10/\ka$. So from now on $U'$ and $N$ are fixed. 

Now suppose that $U_1 \subset U \subset \hat{U} \subset U'$, and let $U_1 = \cup_j B(c_j,\de_1)$, $U = \cup_j B(c_j,\de)$, $\hat{U} = \cup_j B(c_j,\hat{\de})$ i.e., $\de_1 \leq \de \leq \hat{\de} \leq \de'$. Suppose that $z \notin U_1$, $f^k(z) \notin U$ for all $k=1, \ldots, n-1$ and $f^n(z) \in U$. Let now $n_0 > 0$ be the first time for which (\ref{dyadic}) is not valid
with $U_j'$ replaced by $U_j$, the components of $U$. Let $W_k = f^{-k}(U_j)$ be the corresponding preimages of $c_j \in U_j$. By the definition of $n_0$, 
\begin{equation}
  \dist(W_{n_0},Crit) \leq (1/\ka) \diam(W_{n_0}) \leq  (1/\ka) \la_1^{-n_0} \diam(U_j). 
\end{equation}
Let us now consider the condition
\begin{equation} \label{condition}
  (1/\ka) \la_1^{-n_0} \leq \frac{q}{10} \leq \frac{\diam(U_{1,j})}{10 \diam(U_j)},
\end{equation}
  where $U_{1,j} \subset U_j$ is the corresponding component of $U_1$ inside $U_j$ and the second inequality is valid by assumption. We discuss this condition soon. It implies that 
\begin{align}
 \dist(W_{n_0},Crit) &\leq \frac{\diam(U_{1,j})}{10 \diam(U_j)} \diam(U_j) = \frac{1}{10}\diam(U_{1,j}), \quad  \text{ and } \\
\diam(W_{n_0}) &\leq  \la_1^{-n_0} \diam(U_j) \leq  \frac{\ka}{10}\diam(U_{1,j}).
\end{align}
Clearly, this implies that $W_{n_0} \subset U_{1,j} \subset U_1$. If $n_0 \leq n$, this was not allowed, since $z_k \notin U_1$, $1 \leq k \leq n$. Hence (\ref{dyadic}) is valid all the time up until $n$. Therefore, if $w \in f_0^{-n}(U)$, we have, by the distortion estimate (\ref{distortion-est}),
\begin{equation} \label{exp-growth}
  |Df_0^n(w)|  \geq |Df_a^n(z)| C_3^{-n} \geq C_2 \la_1^n,
\end{equation}
where $z$ is the preimage of the corresponding critical point and $C_2$ is the constant form the second Collet-Eckmann condition, and hence does not depend on $U$. 

Let us now discuss the condition (\ref{condition}). Then the condition implies that
\begin{equation} \label{Nq}
n_0 \log \la_1 \geq \log (1/q) - \log \ka + \log 10 \geq \De_1 - \De - \log \ka + \log 10.
  \end{equation}
  Hence this basically forces $n_0-1$, the time when (\ref{dyadic}) is valid, to be bounded below by the difference $\De_1 - \De$. Let now $\hat{N}$ be the largest integer such that (\ref{dyadic}) is valid with $U_j$ replaced by $\hat{U}_j$, the components of $\hat{U}$. Then we can say that $\hat{U}$ depends on $q$ in the following sense. For a fixed $q$ we choose $\hat{U}$ so that the corresponding $\hat{N}$ satisfies (\ref{Nq}), with $n_0$ replaced by $\hat{N}$. Clearly, if $q=1$ we can put $\hat{U} = U'$. 

From a classical result by R. Ma\~n\'e, we have an estimate as follows. If $f_0^k(z) \notin U$ for $k=0,\ldots,n$, then 
\begin{equation} \label{weak-oe}
|Df_0^n(z)| \geq C \la_3^n,
  \end{equation}
  for some constant $C >0$ that depends on $U$ and $\la_3 > 1$ which does not depend on $U$. We may assume that $\la_3 > \la_1$, otherwise diminish $\la_1$ so that this holds. This proves the first statement of the lemma. 

Choose $N_1 >0$ so that outside $U_1$ the orbits $f_a^k(z)$ and $f^k(z)$ follow each other up to $N_1$, i.e. for $k \leq N_1$, and so that $|Df_a^{N_1}(z)| \geq C \ti{\la}_3^{N_1} \geq \ti{\la}_1^{N_1}$, for all $a \in (-\vep,\vep)$. Here $\ti{\la}_3 > 1$ comes from a perturbed version of (\ref{weak-oe}). Since also (\ref{exp-growth}) is valid for small perturbations if we bound the number of iterations by $N_1$, we let $\ti{\la}_1 > 1$ be the corresponding perturbed version of $\la_1 > 1$.  Let us write $n$ as $n = qN_1 + r$, where $r < N_1$. Then, if we assume that $z \notin U_1$, $f_a^k(z) \notin U$ for $k=1,\ldots, n-1$ and $f_a^n(z) \in U$ we get
\begin{equation} \label{uniform-exp}
  |Df_a^n(z)| = |Df_a^r(f^{qN_1}(z))| |Df_a^{N_1}(f^{(q-1)N_1}(z))| \ldots |Df_a^{N_1}(z)| \geq C_2 \ti{\la}_1^n,
\end{equation}
where we used (\ref{exp-growth}) for $|Df_0^r(f^{q N_1}(z))| \geq C_2 \la_1^r$, so that $|Df_a^r(f^{q N_1}(z))| \geq C_2 \ti{\la}_1^r$. The second statement of the lemma follows with $\la = \ti{\la}_1$. 
\end{proof}

The classical outside expansion lemma is obtained by setting $U_1 = U$ in the above lemma, i.e. $q=1$.
From \cite{GS2}, it can be seen that the Lyapunov exponent from the second Collet-Eckmann condition is inherited from the exponent from the ordinary Collet-Eckmann condition. Hence the uniform ``outside exponent'' $\log \ti{\la}_1$, is close to the Lyapunov exponent for the starting map $f_0$ (but likely lower than it), depending on the neighbourhood $U'$. Let us set $\ga_H = \log \ti{\la}_1$.

\section{Parameter-phase distortion} \label{L-parameter-space}

One fundamental result we need is the comparison between space and parameter-derivatives. This has been proved in \cite{BC1,BC2} and many other papers. But for our purposes we need a stronger form of this result due to G. Levin. We use a normalised space, described in \cite{Levin-book}, of maps in $\RR_d$ as follows. We consider the set $\La_{d,\oli{p'}} \subset \RR_d$ of all rational maps of degree $d$ with exactly $p'$ distinct critical points $c_j$ with corresponding multiplicities $m_j$, $1 \leq j \leq p'$, where $\oli{p'} = \{m_1, \ldots, m_{p'} \}$, normalised to that every map $f \in \La_{d,\oli{p'}}$ has the form  
\[
f(z) = \si z + b + \frac{P(z)}{Q(z)}, 
\]
where $\si \neq 0$, and $\deg(P) \leq d-2$, $\deg(Q) \leq d-1$ and where $P$ and $Q$ have no common zeros. By Proposition 8 in \cite{Levin-book}, every $f \in \RR_d$ is conjugate by a M\"obius transformation to some $\ti{f} \in \La_{d,\oli{p'}}$. 
So we can view $\RR_d$ as a union of sets of the type $\La_{d,\oli{p'}}$ up to equivalence by M\"obius transformations. Note that in every such set, critical points do not split. 

We assume that the real analytic family $f_a \in \La_{d,\oli{p'}}$, $a \in (-\vep,\vep)$ around $f_0$ has a tangent vector $\oli{u} \neq \oli{0}$. Hence $f_a(z) = f_0(z) + a u(z) + \OO(a^2)$ for some $u \neq 0$. Let us write $\xi_{n,l}(\oli{a}) = f_{\oli{a}}^n(c_l(\oli{a}))$ , where $\oli{a} = (a_1,a_2,\ldots,a_{p'})$ is a parameterisation of the parameter space $\La_{d,\oli{p'}}$ around $f=f_0$, where $f_0$ corresponds to $(a_1,a_2,\ldots,a_{p'})  = (0,0,\ldots,0)$ and where  $c_j=c_j(a_1,a_2,\ldots,a_{p'})$. In \cite{Levin-book} and \cite{Levin-transv}, it is proven that
the matrix $L$ formed by the numbers
\[
  L(c_l,a_k) = \lim\limits_{n\raw \infty} \frac{\dfrac{\partial \xi_{n,l}}{\partial a_k}(\oli{0})}{Df_0^{n-1}(fc_l)}
\]
is non-degenerate. Let $\oli{u} = (u_1, u_2, \ldots, u_{p'})$ be a tangent vector of unit length, i.e., a vector in $\P( \La_{d,\oli{p'}} )$, and suppose that this is tangent to the family $f_a$ at $a=0$. Then for almost all directions, i.e. tangent vectors, we have that all entries of $L \cdot \oli{u}$ are non-zero, since the set of directions when this is not true is a finite union of sets of co-dimension $1$ in $\P (\La_{d,\oli{p'}})$. This means precisely that, for almost all directions, the limits
\[
\lim\limits_{n \raw \infty}  \frac{\xi_{n,l}'(0)}{Df_0^{n-1}(f_0(c_l))} = \sum_k a_k L(c_l,a_k),
\]
is non-zero for every $l$, where we mean $\xi_{n,l}'(0) = \frac{\ud}{\ud a} \xi_{n,l}(a \oli{u}) \bigr|_{a=0}$. We thus say that the real analytic family $f_a$ around $f_0$ is {\em non-degenerate}, if its tangent vector satisfies this condition, (in addition to the fact that the critical points move analytically).  

We summarise this result as a proposition below, which is a direct consequence of Theorem 1 combined with Corollary 2.1, part (8), in \cite{Levin-transv}.  It is a generalisation of a corresponding result in  \cite{Levin-book} Theorem 1.1. 

\begin{Prop}[Levin] \label{levin}
Suppose that $f$ is a rational map with summable critical points without parabolic cycles such that $J(f) = \hat{\C}$, and suppose that $f_a$ is a non-degenerate real analytic family around $f_0$, $a \in (-\vep,\vep)$. Then for each critical point $c_l(a)$, the limit
\begin{equation}
\lim\limits_{n \raw \infty} \frac{\xi_{n,l}'(0)}{Df^{n-1}(v_l)} = L_l
\end{equation}
exists and is different from $0$ and $\infty$.
\end{Prop}

Indeed, a CE-map has all its critical values summable so the above proposition can be used. We also note that by \cite{GS2} any Collet-Eckmann map different from a flexible Latt\'es map carries no invariant line field on its Julia set. We now use this result, to make small perturbations. 

\begin{Lem} \label{transv1}
Assume that $f_0$ satisfies the CE-condition with exponent $\ga$. For any $0 < \ga_1 < \ga$ and $0 < q < 1$, there exists $N > 0$ and $\vep > 0$ such that if $f_a$, $a \in (-\vep,\vep)$ satisfies the CE-condition up to time $m \geq N$ with exponent $\ga_1$, we have
 \[
\biggl| \frac{\xi_{m,l}'(a)}{Df_a^{m-1}(v_l(a))} - L_l \biggr| \leq q |L_l|,
\]
for every $l$. 
  \end{Lem}
  \begin{proof}
According to Theorem 1 in \cite{Levin-transv}, we have for $a=0$,
\[
  \lim\limits_{n \raw \infty} \frac{\xi_{n,l}'(0)}{Df^{n-1}(v_l(0))} = \sum_{n=0}^{\infty} \frac{\partial_a f_0(\xi_{n,l}(0))}{Df_0^n(v_l(0))} = L_l.
\]
Let us put $\xi_{m.l}(a) = \xi_m(a)$ and $L_l=L$. 
The reader may verify that for small perturbations $a$ close to $0$,
\[
\frac{\xi_m'(a)}{Df_a^{m-1}(v_l(a)} = \sum_{n=0}^{m-1} \frac{\partial_a f_a(\xi_n(a))}{Df_a^n(v_l(a))}.
\]
We have that $|\partial_a f_a| = |\partial_a f(z,a)|$ is bounded by some constant $B > 0$. We choose $N>0$ so that the series
\begin{equation} \label{tail}
\sum_{n=N+1}^{\infty}  \frac{B}{Ce^{\ga_1 n}}  \leq \min_{l} (q |L_l|/4).
\end{equation}
By continuity, there exists some $\vep > 0$ such that if $a \in (-\vep,\vep)$ then
\[
\biggl|  \sum_{n=0}^{N}  \frac{\partial_a f_a(\xi_{n,l}(a))}{Df_a^n(v_l(a))} - L_l \biggr| \leq q |L_l|/2. 
\]
Since $f_a$ is assumed to satisfy the CE-condition with exponent $\ga_1$, by (\ref{tail}) we get that the tail satisfies
\[
\biggl|   \sum_{n=N+1}^{m}  \frac{\partial_a f_a(\xi_{n,l}(a))}{Df_a^n(v_l(a))} \biggr| \leq q |L_l|/2,
\]
for all $a \in (-\vep,\vep)$ and all $m \geq N$. This finishes the lemma. 
\end{proof}

When we use this lemma we want to choose $N$ and $\vep$ so that the above lemma is valid for $\ga_L = (1/6) \min(\ga_0,\ga_H)(1-\tau)$, where $0 < q < 1$ is small, and where $\tau$ is a constant, $0 < \tau < 1$.

\section{Weak parameter dependence and weak distortion} \label{paramindep}

We will later see that the expansion of the space derivative induces a great deal of parameter independence. This follows {\em a posteriori} from the Main Distortion Lemma \ref{main-distortion} and the Starting Lemma \ref{startlemma}, but to start we now prove a weaker statement.

\begin{Lem} \label{weak-distortion}
Let $N$ be as in Lemma \ref{transv1} and let $\ga_1 \geq (1/4) \min(\ga_H, \ga_0) (1-\tau) > 0$. Suppose that $a,b \in (-\vep,\vep)$, where $(-\vep,\vep)$ is a parameter interval around $f_0$ for which $f_a$ is the real analytic family we are considering. If $\vep > 0$ is sufficiently small we have the following. Suppose that we have;
\begin{itemize}
\item[i)] for all $n \leq N$, that $|Df_a^n(v_l(a))| \geq C_1 e^{\ga_1 n}$, and for $k \leq k_1$ for some $k_1 \geq 0$, that $|Df_a^k(\xi_{N,l}(a))| \geq C_2 e^{\ga_1 k}$,
\item[ii)] for all $n \leq N + k_1$, if $\xi_{n,l}(a), \xi_{n,l}(b) \notin U$, then $|\xi_{n,l}(a)-\xi_{n,l}(b)| \leq S$  and, if  $\xi_{n,l}(a) \in U$ or $\xi_{n,l}(b) \in U$ (or both), then
  \[
    |\xi_{n,l}(a)-\xi_{n,l}(b)| \leq \dist(\xi_{n,l}(c),Crit_c)/ ( \log(\dist(\xi_{n,l}(c),Crit_c)))^2,
  \]
  where $c \in \{ a,b \}$ is such that $\dist(\xi_{n,l}(c),Crit_c)$ is minimal, 
\item[iii)] that $f_a$ is slowly recurrent up until time $N+k_1$.
\end{itemize}

Then there exists $Q > 1$ (arbitrarily close to $1$, if $N$ is large enough), and $\ga_2 > 0$ (arbitrarily close to but slightly smaller than $\ga_1$), such that
\begin{align} \label{ep-strx}
  |\xi_{N+k,l}(a) -\xi_{N+k,l}(b) | &\sim_{Q^k} |Df_a^k(\xi_{N,l}(a))| |\xi_{N,l}(a) - \xi_{N,l}(b)|  \text{ and }  \nonumber\\
  |\xi_{N+k,l}(a) -\xi_{N+k,l}(b) | &\geq |\xi_{N,l}(a) - \xi_{N,l}(b)| C_2 e^{\ga_2 k},
\end{align}
for any $k \leq k_1$, where $\ga_2 > 0$ is slightly smaller than $\ga_1 > 0$.
\end{Lem}

\begin{proof}
Since we assume that the critical points $c_l(a)$ move analytically in $a$ we have
\[
c_l(a) = K_l a^{k_l} + \OO(a^{k_l+1}).
\]

Let us fix $l$ and consider $\xi_{n,l}(a) =\xi_n(a)$. If we consider a sufficiently small parameter interval $(-\vep,\vep)$ centred at $a=0$ corresponding to $f_0$, then, by bounded distortion, we can make $\vep$ so small so that we have, for any two points $a,b \in (-\vep, \vep)$, 
\[
|\xi_{N}(a) - \xi_{N}(b) \sim_2 |\xi_{N}'(c)||a-b|,
\]
for any $c \in (-\vep,\vep)$. 
From the assumption $|Df_a^N(v_l(a))| \geq C_1 e^{\ga_1 N}$ we see that we may choose $\vep > 0$ small enough to get
$|(f_b^N)'(v_l(b))| \geq C_1 e^{\ga_2 N}$ for $b \in \om = (-\vep,\vep)$ for some $\ga_2 > 0$ slightly smaller than $\ga_1$.
From Lemma \ref{transv1} we now get, with $q \leq 1/2$ and $L=|L_l|$, for any $c \in (a,b)$, 
\begin{multline}  \label{ep-str1}
  |\xi_{N}(a) - \xi_{N}(b) | \sim |\xi_{N}'(c)| |a-b| \geq q L |a-b| |Df_{c}^{N-1}(v_l(c))|  \\
                            \geq C_1 q L e^{\ga_2 (N-1)} |a-b|  
                            = C_1' e^{\ga_2 N} |a-b|,   
\end{multline}
where $C_1' = C_1 q L e^{-\ga_2}$.

During the first $N$ iterates, (\ref{ep-str1}) implies that for $a$ and $b$ close to $0$ we have
\[
  |c_l(a) - c_l(b)| \leq 2K_l k_l a^{k_l-1} |a-b| \leq C a^{k_l-1} |\xi_{N,l}(a) - \xi_{N,l}(b)| e^{-\ga_2 N}, 
\]
for some constant $C$. It follows that
\begin{equation} \label{slow-movement}
  |\xi_{N,l}(a) - \xi_{N,l}(b)| \gg |c_l(a) - c_l(b)|
\end{equation}
for all critical points.

Suppose that, for all $0 \leq j \leq k \leq k_1 - 1$, we have
\begin{equation} \label{ep-str2}
|\xi_{N+j}(a) - \xi_{N+j}(b) | \geq C_2 e^{\ga_2 j} |\xi_N(a) - \xi_N(b)|.
  \end{equation}
We may assume that $\ga_2 < \ga_1 < \ga_0$. Combining (\ref{ep-str2}) and (\ref{ep-str1}) we conclude that (\ref{slow-movement}) holds for $N$ replaced by $N+j$. 
  
For the proof, put $\xi_n(a) = \xi_{n,l}(a)$. Since we assume that the orbit of $w=\xi_n(a)$ stays close to $z=\xi_n(b)$ we have a distortion estimate
\[
\frac{1}{C} \leq \frac{|Df_a(w)|}{|Df_b(z)|} \leq C,
\]
for some constant $C \geq 1$. This constant can be arbitrarily close to $1$ if $|z-w| \leq S$ and $S$ is small enough (for $z,w \notin U$) and $|z-w| \leq e^{-r}/r^2$ (if $\dist(z,Crit) \in (e^{-r-1},e^{-r})$). Hence iterating this we get 
\[
|Df_b^k(\xi_N(b))| \geq \frac{1}{C^k} |Df_a^k(\xi_N(a))| \geq C_2 e^{\ga_2 k}, 
\]
for some $\ga_2 > 0$ slightly smaller than $\ga_1$ (at least we may assume, for instance, that $\ga_2 \geq (9/10) \ga_1$).  


With $B = \sup |\partial_a f_a|$, using (\ref{ep-str1}) and (\ref{ep-str2}), there is some $Q_0 > 1$ such that
\begin{multline}
  |\xi_{N+k+1}(a) - \xi_{N+k+1}(b) |  \\
  \geq \bigl| |f_a(\xi_{N+k}(a)) - f_a(\xi_{N+k}(b))|  - |f_a(\xi_{N+k}(b)) - f_b(\xi_{N+k}(b))| \bigr| \nonumber \\
                                 \sim_{Q_0} |Df_a(\xi_{N+k}(a))| |\xi_{N+k}(a) - \xi_{N+k}(b)|  - |\partial_a f_a (\xi_{N+k}(a))| |a-b| \nonumber \\
                  \geq |Df_a(\xi_{N+k}(a))| |\xi_{N+k}(a) - \xi_{N+k}(b)|  - \frac{B}{C_1'C_2} e^{-\ga_2 (N+k)} |\xi_{N+k}(a) - \xi_{N+k}(b)| \nonumber \\
                                     = \biggl( |Df_a(\xi_{N+k}(a))| - \frac{B}{C_1'C_2} e^{-\ga_2 (N+k)} \biggr) |\xi_{N+k}(a) - \xi_{N+k}(b)|. \label{ind-k}
\end{multline}
It is easy to check that a reverse inequality also holds. Note that $Q_0$ can be chosen arbitrarily close to $1$ if $N$ is large enough and $S = \de \vep_1$ is small enough (i.e. $\vep_1$ small enough). 
Repeating this $k$ more times we get
\begin{multline}
  |\xi_{N+k+1}(a) - \xi_{N+k+1}(b) | \\
  \sim_{Q_0^{k+1}} |Df_a^{k+1}(\xi_N(a))| \prod_{j=0}^{k} \biggl( 1 - \frac{Be^{-\ga_2 (N+j)}}{C_1' C_2 |f_a'(\xi_{N+j}(a))|} \biggr) |\xi_N(a) - \xi_N(b)|.
  \end{multline}

Now we use that $f_a$ is slowly recurrent (actually this implies that $f_b$ is also slowly recurrent since $\xi_n(a)$ stays close to $\xi_n(b)$). But we have assumed from the beginning that the exponent in basic assumption (which is inherited from the slow recurrence condition) satisfies $\al \leq (1/(400K\Ga)) \min(\ga_H,\ga_0)(1-\tau) \ll \ga_1$. Then $\al$ is also much smaller than $\ga_2$, and, if $N$ is sufficiently large, for some constant $C > 0$,
\[
|Df_a(\xi_{N+j}(a))| \geq Ce^{-K \alpha (N+j)} \gg e^{-\ga_2 (N+j)}, \text{ for all $0 \leq j \leq k_1$.} 
\]
We see that the sum $\sum_{j=0}^{\infty} e^{-(\ga_2 - \al)(N+j)}$ can be made as small as we like. Therefore, the product
\[
 \prod_{j=0}^{k} \biggl( 1 - \frac{Be^{-\ga_2 (N+j)}}{C_1' C_2 |f_a'(\xi_{N+j}(a))|} \biggr) \geq  \prod_{j=0}^{\infty} (1 - Ce^{-(\ga_2  +\al) (N+j)}) > \frac{1}{Q_1},
\]
for some $Q_1 > 1$ (independent of $k$). Therefore, 
\begin{equation}
  |\xi_{N+k+1}(a) - \xi_{N+k+1}(b) | \sim_{Q_0^{k+1} Q_1} |Df_a^{k+1}(\xi_N(a))|  |\xi_N(a) - \xi_N(b)| \\ \label{endpoint-stretch}
  \end{equation}
  Since $|Df_a^{k+1}(\xi_N(a))| \geq C_2 e^{\ga_1 (k+1)}$ we have $Q_0^{k+1}Q_1 |Df_a^{k+1}(\xi_N(a))| \geq C_2 e^{\ga_2 (k+1)}$, for some $\ga_2 > 0$ slightly smaller than $\ga_1$, given that $Q_0$ and $Q_1$ are sufficiently close to $1$. Hence we have (\ref{ep-str2}) satisfied with $k$ replaced by $k+1$ and we can continue the same argument and obtain (\ref{ep-str2}) up until $k_1$. This settles both claims with $Q^k=Q_0^kQ_1$.
\end{proof}


We can easily get a little more general statement. If $f_a$ satisfies the CE-condition and is slowly recurrent up until time $N + k_1$, we can use (\ref{ind-k}), and the following arguments to obtain 
    \begin{equation} \label{weak-distortion2}
|\xi_n(a) - \xi_n(b)| \sim_{Q^j} |Df_a^j(\xi_{n-j}(a))||\xi_{n-j}(a) - \xi_{n-j}(b)|,
\end{equation}
if $n - j \geq N$ and $n \leq N + k_1$. The details are left to the reader.

\begin{Rem} \label{paramindep-rem}
  We have seen that the parameter dependence is inessential as long as the derivative of $|f_a^n(v_l(a))|$ grow with a certain Lyapunov exponent $\ga_1$. We call this the {\em weak parameter dependence property}. We will also require that the Lyapunov exponent never goes below a certain ``critical'' level, which is related to the ``intermediate level'' $\ga_I = (1/3) \min(\ga_H, \ga_0)(1-\tau)$. Since $\ga_L = (1/6) \min(\ga_H, \ga_0) (1-\tau) < (1/4) \min(\ga_H, \ga_0) (1-\tau)$, then $\ga_C = (1/4) \min(\ga_H, \ga_0) (1-\tau)$ (the critical exponent) as a lower bound for $\ga_1$ will do. We also let $\ga_B = (3/4) \min(\ga_H, \ga_0) (1-\tau)$. This $\ga_B$ is the Lyapunov exponent that we want to keep at the end. 
\end{Rem}

\section{Distortion and expansion during the bound period}  \label{bound-distortion}

We use the following notations. Below $\om$ is assumed to be an interval and a partition element according to Definition \ref{partition-element}. 
\begin{Def}
  We say that $a \in \EE_{n,l}(\ga)$ if
  \begin{align}
    |Df_a^n(v_l(a))| &\geq C_0 e^{\ga k}, \text{  for all $ k \leq n-1$, and} \label{exp1} \\
     |Df_a^n(v_j(a))| &\geq C_0 e^{\ga k}, \text{  for all $ k \leq (2K \al/\ga_I) n$, and all $j \neq l$.} \label{exp2}
    \end{align}
    We say that $a \in \BB_{n,l}$ if
    \begin{align}
      \dist(\xi_{k,l}(a),Crit_a) &\geq K_b e^{-2\al k}, \text{ for all $k \leq n$ and}
      \label{ba1} \\
      \dist(\xi_{k,j}(a),Crit_a) &\geq K_b e^{-2\al k}, \text{ for all $k \leq (2K \al/\ga_I) n$ and all $j \neq l$.} \label{ba2}
    \end{align}
We say that $\om \subset \EE_{n,l,\star}(\ga)$  if (\ref{exp1}) holds and (\ref{exp2}) holds with $(2K \al/\ga_I)$ replaced by $4(K \al/\ga_I)$. We say that $\om \subset \BB_{n,l,\star}$ if (\ref{ba1}) holds and (\ref{ba2}) holds with $(2K \al/\ga_I)$ replaced by $4(K \al/\ga_I)$.
    \end{Def}

Note that $\om_0 \subset \EE_{N,l}(\ga) \cap \BB_{N,l}$ for all $l$ for some $\ga$ close to $\ga_0$. The definitions above is tailored so that if an interval belongs to $\EE_{n,l}(\ga)$ or $ \BB_{n,l}$ then we can use the binding information for the other critical points up until some fraction $2K\al/\ga_I$ of the time $n$. The star is added to be able to use the binding information longer and continue the parameter-exclusion construction up until $2n$. 

To prove bounded distortion, we will frequently make use of the following lemma, which is standard.

\begin{Lem}  \label{xexp}
Given complex numbers $z_1, \ldots, z_n$ we have
\[
\biggl| \prod_{j=1}^n z_j  - 1 \biggr| \leq -1 + \exp{\sum_{j=1}^n |z_j - 1|}. 
\]
\end{Lem}

Expanding $f$ in Taylor series near a critical point $c$ gives
\[
f_a(z) = A(z-c)^k + \OO((z-c)^{k+1}), \qquad Df_a(z) = Ak(z-c)^{k-1} + \OO((z-c)^k),
\]
where $A$ is analytic in the parameter $a$. If $z$ and $w$ are close to $c$ and $|z-c| \sim |w-c|$, we get,
\begin{multline} \label{T-exp}
  Df_a(z) - Df_a(w) = Ak(z-w)( (z-c)^{k-2} + (z-c)^{k-3}(w-c) + \ldots \\
  + (w-c)^{k-2}  + \OO((z-c)^{k-1})). 
\end{multline}
Hence, 
\[
\sum_{j=1}^{n} \frac{|Df_a(\xi_j(a)) - Df_a(\xi_j(b))|}{|Df_a(\xi_j(b))|} 
\sim_{2k} \sum_{j=1}^{n} \frac{|\xi_j(a)) - \xi_j(b)|}{\dist(\xi_j(b),Crit_b)}.
\]
if $z$ and $w$ are sufficiently close to $Crit(f)$. In Section \ref{strong-distortion} we will allow the parameter to vary as well, and have to go a bit further. 

\begin{Lem}[Distortion during the bound period] \label{bound-dist}
Let $\vep' > 0$. Then if $\de'=e^{-\De'}$ is sufficiently small and $N$ sufficiently large, the following holds.  Let $z=\xi_{\nu,l}(a)$ be a free return into $U_i'$, $\nu \geq N$, where $a \in \EE_{\nu,l}(\ga) \cap \BB_{\nu,l}$ for some $\ga \geq \ga_I$. Then we have, for all $w$ on the line segment between $f_a(z)$ and $\xi_{1,i}(a) = v_i(a)$, 
\[
\biggl| \frac{Df_a^j(w)}{Df_a^j(v_i(a))} - 1 \biggr| \leq \vep',
\]
for $j \leq p$, where $p$ is the length of the bound period for $z$.
\end{Lem}

\begin{proof}
We first prove the lemma for $w=f_a(z)$. Let $\dist(\xi_{\nu,l}(a), Crit_{a}) \sim_{\sqrt{e}} e^{-r}$ where $\xi_{\nu,l}(a) = z$ and put $z_j = f_a^j(z)$ and $\xi_{j,i}(a) = \xi_j(a)$. Following the discussion preceding the lemma, we estimate, for $\nu \geq N$, the sum
\[
\sum_{j=1}^{p} \frac{|Df_a(z_j) - Df_a(\xi_{j}(a))|}{|Df_a(\xi_{j}(a))|}  \leq C \sum_{j=1}^{p}  \frac{|z_j - \xi_{j}(a)|}{\dist(\xi_{j}(a),Crit_a) }.  
\]
The last sum can be divided into two subsums $[1,J] \cup [J+1,p]$ where $J = dr/(10(2 \al + \Ga))$, where $d$ the degree of $f_0$ at $c_k$, and $\Ga = \sup_{a \in (-\vep,\vep), z \in \hat{\C}} \log |f_a'(z)|$. Assuming that the basic approach rate assumption holds, the first sum an be estimated as
\[
  \sum_{j=1}^J \frac{|z_1 - \xi_1(a)|e^{\Ga j}}{K_b e^{-2 \al j}} \leq \sum_{j=1}^J CK_b^{-1}e^{-dr} e^{(\Ga + 2 \al)j } \leq
  \sum_{j=1}^J Ce^{-(9/10)dr} \leq Ce^{-9\De'/10}. 
 \]
 The second sum can be estimated using the definition of the bound period (remember $\be=\al$), 
\[
 \sum_{j=J+1}^p \frac{|z_j - \xi_j(a)|}{\dist(\xi_j(a),Crit_a) } \leq C \sum_{j=J+1}^p e^{-\al j} \leq C e^{-\al \frac{dr}{10(2 \al  + \Ga)}}.  
\]
We see that both sums can be made arbitrarily small if $\De'$ is large enough. This finishes the case $w=f_a(z)$.

It is easy to see that the same must hold one the line segment between $f_a(z)$ and $v_i(a)$. Let $p' \leq p$ be the least bound period for all such $w$ on this line. That means that up until $j=p'$ the distortion estimate holds for all $w$ on the line. But since $\vep'$ may be chosen very small this means that the image of the line  under $f_a^{p'}$ is an almost straight line too. It follows that the corresponding $w$ for $p'$ has to be $z$ in fact, so $p=p'$.  
\end{proof}

Note that the condition on $U'$ in the above lemma is only depending on the starting family of functions $f_a$, $a \in (-\vep,\vep)$. There is also a condition on $U'$ in Lemma \ref{oel}.  

\begin{Lem} \label{boundexp}
Suppose that $\xi_{\nu,l}(a)$ is a return into $U_i'$ and that $a \in \EE_{\nu,l}(\ga) \cap \BB_{\nu,l}$ for some $\ga \geq \ga_I$. Then if $N$ is large enough and $p$ is the length of the following bound period we have,  
\[
|Df_a^{p}(\xi_{\nu,l}(a))| \geq e^{\frac{\ga}{2d_i} p}, 
\]
where $d_i$ is the degree of $f$ at $c_i$.

Moreover, if $\dist(\xi_{\nu,l}(a),Crit_{a}) \sim_{\sqrt{e}} e^{-r}$, then 
\[
\frac{d_i r }{2 \Ga} \leq p \leq \frac{2 d_i r}{\ga}.
\]
In particular, $p \leq 2 \al d_i \nu /\ga$, where $\al$ is the exponent in the basic assumption and $\Ga = \sup\limits_{a \in (-\vep,\vep), z \in \hat{\C}} \log |Df_a(z)|$. 
\end{Lem}
\begin{proof}
Put $D_j = |Df_a^j(\xi_{\nu,l}(a))|$ and $E_{j} = |Df_a^{j}(\xi_{\nu+1,l}(a))|$ for some $a \in \om$. We have $D_1 \geq CK_b e^{-2\al K \nu}$, since $a \in \BB_{n,l}$ for some constant $C$. Moreover, for $1 \leq j \leq p-1$, we can use Lemma \ref{bound-dist} to prove that $E_j \geq (C_0/2) e^{\ga j}$ since $a \in \EE_{\nu,l}(\ga)$. Hence the derivative
\[
|Df_a^{\nu+j}(v_l(a))| \geq (C_0/2) C K_b C_0 e^{(\ga-2\al K) (\nu+j)} \geq C_0e^{\ga' (\nu+j)} , \text{ for $j \leq p$},
\]
where $\ga' \geq \ga -4 \al K \geq \ga_C$, provided $N$ is large enough (recall $\nu \geq N$). We can also use Lemma \ref{bound-dist} to get
the following distortion estimate, for some $C > 1$ (close to $1$),
\[
|\xi_{\nu+j,l}(a) - \xi_{j,i}(a)| \sim_{C} |Df_a^j(\xi_{\nu,l}(a))||\xi_{\nu,l}(a) - \xi_{0,i}(a)|,
\]
for $j \leq p+1$. Suppose that $|\xi_{\nu,l}(a) - \xi_{0,i}(a)| \sim_2 e^{-r}$.  We know from the definition of the bound period and the basic assumption, that
  \begin{equation} \label{bound-length}
D_{p+1} e^{-r} \geq \frac{1}{4C} \dist(\xi_{p+1,i}(a),Crit_a) e^{-\al (p+1)} \geq  \frac{1}{4C} K_be^{-2\al (p+1) - \be (p+1)}.
    \end{equation}
Also we have, for some $\ka_1 \geq 1$,
\[
  D_{p+1} e^{-r} \sim_{\ka_1} E_{p} e^{-rd_i},
\]
and so
\begin{align}
  e^{-r(d_i-1)} &\sim_{\ka_1} \biggl( D_{p+1} e^{-r} \biggr)^{\frac{d_i-1}{d_i}} E_p^{-\frac{d_i -1}{d_i}}  \nonumber \\
  &\geq \biggl( \frac{K_b}{4C} \biggr)^{\frac{d_i-1}{d_i}}  e^{-(2 \al + \al) (p+1) \frac{d_i -1}{d_i} }   E_{p}^{-\frac{d_i-1}{d_i}} .
\end{align}
Now we can use that $2\al + \be$ is very small compared to $\ga \geq \ga_I$. We get, 
\begin{align}
  D_{p+1} &\sim_{\ka_1} e^{-r(d_i-1)} E_{p} \nonumber \\
          &\geq \biggr( \frac{K_b}{4C} \biggl)^{\frac{d_i-1}{d_i}} E_{p}^{\frac{1}{d_i}} e^{-3 \al (p+1) \frac{d_i-1}{d_i} } \nonumber \\
  &\geq  \biggl( \frac{C_0}{2}\biggr)^{\frac{1}{d_i}} \biggl(\frac{K_b}{4C} \biggl)^{\frac{d_i-1}{d_i}} e^{\frac{\ga}{d_i} p - 3 \al (p+1)} \geq e^{\frac{p}{2d_i}\ga},
\end{align}
if $\nu$ is sufficiently large. Since $D_p = D_{p+1}/|Df_a(\xi_{\nu+p}(a))|$, with minor modifications it is easy to see that the same estimate holds for $D_p$. 

To prove the second claim, we note that from (\ref{bound-length}), the slow recurrent condition and the fact that  $|Df^{\nu}(v_l(a))| \leq e^{\nu \Ga}$ we get that, for some very small $\al > 0$ in comparison to $\ga$,
\[
e^{\Ga (p+1)}e^{-d_i r} \geq E_p e^{-d_i r} \geq   \frac{K_b}{4C} \ka_1^{-1} e^{-3 \al (p+1)},
\]
which gives the left inequality if $\nu \geq N$ is large enough. To prove the right inequality, we note that the spherical distance $\dist(\xi_{\nu,l}(a),Crit_a)$ is bounded from above. By the definition of the bound period (now we are considering the time $p$ iterates from the return into $U$), and the fact that we also have $E_{p-1} e^{-d_i r} \sim_{\ka_1} D_p e^{-r}$,
\[
  (C_0/2)e^{\ga (p-1)} e^{-d_i r} \leq E_{p-1} e^{-d_i r} \leq 4C \ka_1  e^{-\al p} \dist(\xi_{p,i}(a),Crit_a). 
\]
and the right inequality follows. 
\end{proof}

The above lemma gives a quite substantial amount of increase of the derivative during the bound period, even if there is a loss in the first iterate. We can also see that under all circumstances, 
\[
|\xi_{\nu+p}(a) - \xi_{\nu+p}(b)| \geq |\xi_{\nu}(a) - \xi_{\nu}(b)|.
\]


\section{Strong distortion} \label{strong-distortion}

Our aim now is to use weak distortion and prove that we actually have something stronger, namely, for some small $\vep' > 0$,
\begin{equation} \label{distortion0}
\biggl| \frac{Df_a^{n}(v_l(a))}{Df_b^{n}(v_l(b))}  - 1 \biggr|  \leq \vep',
\end{equation}
for all $a,b \in \om$ where $\om$ is a partition element according to Definition \ref{partition-element}. We will also make use of the preliminary discussion in Section \ref{bound-distortion}. Let us first see a geometrical consequence of (\ref{distortion0}). By Lemma \ref{transv1} we have, for some small $0 < q < 1$, 
\begin{equation} \label{ddist}
\biggl| \frac{\xi_{n,l}'(a)}{Df_a^{n-1}(v_l(a))} - L_l \biggr| \leq q|L_l|
  \end{equation}
  for $n \geq N$  as long as $f_a$ satisfies the CE-condition with some exponent at least $\ga_L$ and where $N > 0$ is as in Lemma \ref{transv1}. So combining (\ref{distortion0}) and (\ref{ddist}) we get
\begin{equation} \label{good-geometry}
\biggl| \frac{\xi_{n,l}'(a)}{\xi_{n,l}'(b)} - 1 \biggl| \leq \ti{\vep}, \quad \text{ for all $a,b \in \om$},
\end{equation}
where $\ti{\vep} > 0$ is arbitrarily small given that $\vep'$ and $q$ are small enough. This means that the curve $\xi_{n,l}(\om)$ is almost straight, which will be important when we make partitions at returns. 

From now on, let us fix $l$ and write $\xi_{n,l}(a) = \xi_n(a)$. In the beginning we are going to follow orbits close to the original orbit $\xi_n(0)$, and then it is rather easy to see that nearby orbits also satisfy the CE-condition, but when considering nearby parameters $a$ close to $0$, after a long time we have to keep track of the derivative $Df_a^n(v_l(a))$, since the orbit of $\xi_n(a)$ and $\xi_n(0)$ become more or less independent.

Choose some small $\vep > 0$ and suppose that $\om \subset (-\vep,\vep)$. For $a,b \in \om$, consider (\ref{distortion0}). The distortion during the first $N$ iterates can be made arbitrarily small if the perturbation $\vep$ is small enough. So we only need to consider iterates after $N$, and hence focus on proving:
\begin{equation} \label{distortion1}
\biggl| \frac{Df_a^{n-N}(\xi_N(a))}{Df_b^{n-N}(\xi_N(b))}  - 1 \biggr|  \leq \vep'.
\end{equation}
The main task is to prove this stronger form of the space distortion. By Lemma \ref{xexp}, the distortion estimate (\ref{distortion1}) follows if we prove that    
\begin{equation} \label{distortion2}
\sum_{j=0}^{n-N-1} \biggl| \frac{Df_a(\xi_{N+j}(a)) - Df_b(\xi_{N+j}(b)}{Df_b(\xi_{N+j}(b))} \biggr| \leq \vep'',
\end{equation}
where $\vep' \raw 0$ as $\vep'' \raw 0$.

With $z=\xi_j(a)$ and $w=\xi_j(b)$, we have
\[
|Df_a(\xi_j(a)) - Df_b(\xi_j(b))| \leq |Df_a(z) - Df_a(w)| + |Df_a(w) - Df_b(w)|. 
  \]
We also see that, for some $a^* \in [a,b]$,
\begin{equation}
  |Df_a(w) - Df_b(w)| \leq  |a-b| |\partial_a Df_{a^*}(w)| \leq C |\xi_j(a) - \xi_j(b)| e^{-\ga_2 j} ,
\end{equation}
for some constant $C > 0$ since $\partial_a Df(z)$ is bounded. 
  
If $c$ is a critical point, using that $|z-c| \sim |w-c|$, for $z=\xi_j(a)$, $w=\xi_j(b)$, we get, using the Taylor expansion of $f$ near $c$, see  (\ref{T-exp}), that
\[
\sum_{j=1}^{n} \frac{|Df_a(\xi_j(a)) - Df_b(\xi_j(b))|}{|Df_b(\xi_j(b))|} 
\sim_{2k} \sum_{j=1}^{n} \frac{|\xi_j(a)) - \xi_j(b)|}{\dist(\xi_j(b),Crit_b)}.
\]
if $z$ and $w$ are sufficiently close to $Crit(f)$. We will therefore estimate the sum
\begin{equation} \label{sum-dist}
\hat{S} = \sum_{j=N}^{n} \frac{|\xi_{j}(a) - \xi_{j}(b)|}{\text{dist} (\xi_{j}(b), Crit_b) }.
\end{equation}

\begin{Lem} \label{adjacent-returns}
If $N$ is large enough we have the following. Suppose that $\nu_k \geq N$ is a return time and that $\xi_{\nu_k,l}(a)$ is a free return into $U'$ (essential or inessential or a pseudo return), $a \in (-\vep,\vep)$. Moreover, we suppose that $a \in \EE_{\nu_k,l}(\ga) \cap \BB_{\nu_k,l}$, where $\ga \geq \ga_I$. Then until the next free return, we have, 
\[
|Df_a^{\nu_{k+1}}(v_a)| \geq e^{\ga_1 \nu_{k+1}},
\]
where $\ga_1 \geq (9/10) \min(\ga, \ga_H)$. 
\end{Lem}

\begin{proof} 
During the bound period $p_k$ starting directly after the return $\nu_k$, we see from Lemma \ref{boundexp} that 
\[
|Df^{\nu_k + p_k}(v_l(a))| \geq C_0e^{\ga \nu_{k}} e^{\frac{\ga}{2K} p_k},
\]
for each $a \in \om$. Moreover, note that $p_k \leq (2 K \al/\ga) \nu_k$ from Lemma \ref{boundexp}.  After that the free period starts, and by the outside expansion Lemma \ref{oel} we get
\[
|Df^{\nu_{k+1}}(v_l(a))| \geq C_0 C' e^{\ga \nu_{k}} e^{\frac{\ga}{2K} p_k} e^{\ga_H (\nu_{k+1} - (\nu_k + p_k))} \geq e^{\ga_1 \nu_{k+1}},
\]
for some $\ga_1 \geq (9/10) \min(\ga,\ga_H) $ if $N$ is large enough. 
\end{proof}

If we consider a return of $\xi_n(\om)$, where $\om$ is a partition element, we have seen by Lemma \ref{weak-distortion}, that we may disregard from the parameter dependence inside $\om$ as long as the space derivative grows exponentially. The above lemma ensures that $\ga_1 \geq \ga_C$. So, by the weak parameter dependence property, we have
\[
|\xi_{\nu'}(a) -\xi_{\nu'}(b)| \sim_{Q^{\nu'-\nu}} |Df_a^{\nu'-\nu}(\xi_{\nu}(a))||\xi_{\nu}(a) -\xi_{\nu}(b)|
\]
for all $a,b \in \om$. Since $|Df_a^{\nu'-\nu}(\xi_{\nu}(a))| Q^{-(\nu'-\nu)}$ is much greater than $1$ ($\log Q < \al \ll \ga_1$), it follows that two orbits $\xi_{n,l}(a)$ and $\xi_{n,l}(b)$ repel each other up to some large scale or until the next return takes place. We get the following lemma. 
\begin{Lem} \label{double}
Suppose that $\xi_{\nu,l}(\om)$ is a return (inessential or essential or a pseudo return) and that $a,b \in \om$, $\om \subset \EE_{\nu,l}(\ga) \cap \BB_{\nu,l}$ is a partition element, and $\ga \geq \ga_I$. Then if $\nu'$ is the next free return time. 
\[
|\xi_{\nu'}(a) -\xi_{\nu'}(b)| \geq 2  |\xi_{\nu}(a) -\xi_{\nu}(b)|.
\]
\end{Lem}

Next, we prove the Main Distortion Lemma, which is our main object in this section.

\begin{Lem}[Main Distortion Lemma] \label{main-distortion}
Let $\vep' > 0$. Then if $N$ is sufficiently large we have the following. Let $\om \subset \EE_{\nu}(\ga) \cap \BB_{\nu,l}$ be a partition element for some $\ga \geq \ga_I$ and suppose that $\nu \geq N$ is a return time or does not belong to a bound period. Then we have, until the next free return $\xi_{\nu',l}(\om)$, a bound on the distortion if $\om$ is still a partition element at time $n$, namely, 
\[
\biggl| \frac{Df_a^{n}(v_l(a))}{Df_b^{n}(v_l(b))} - 1 \biggl| \leq \vep', \quad \text{for all $a,b \in \om$}
\]
if $\nu \leq n \leq \nu'$. 
\end{Lem}
\begin{proof}
By Lemma \ref{adjacent-returns} the CE-condition is fulfilled with exponent $\ga_1 \geq (9/10)  \min(\ga, \ga_H)$ up until the next free return. Now, $\ga_1 \geq \ga_C$ so we can repeatedly use the weak parameter dependence property. Let us assume that $\nu$ is a return time. If not, replace $\nu$ with the latest return time before $\nu$. 

Put $\xi_{n,l}(a) = \xi_n(a)$. We want to estimate the sum
\begin{equation} \label{sum}
\sum_{j=1}^{n} \frac{|\xi_j(a)) - \xi_j(b)|}{\dist(\xi_j(b),Crit_b)}.
\end{equation}

First we look at the contribution from the bound periods. We want to estimate the sum
\[
\sum_{j=0}^{p} \frac{|\xi_{\nu + j}(a) - \xi_{\nu +j}(b)|}{\dist(\xi_{\nu +j}(b), Crit_b)}.
\]

Since $|\xi_{\nu}(a) - \xi_{\nu}(b)| \sim_2 e^{-r}/r^2$ and $\dist(\xi_{\nu}(b), Crit_b) \sim_{\sqrt{e}} e^{-r}$, the first term ($j=0$) contributes $\sim 1/r^2$. 

To estimate the other terms ($j > 0$), we use the weak parameter dependence property to get
\[
  |\xi_{\nu+j}(a) - \xi_{\nu+j}(b)| \sim_{Q^j} |Df^j(\xi_{\nu}(a))| |\xi_{\nu}(a) - \xi_{\nu}(b)| \sim_2 |Df^j(\xi_{\nu}(a))| e^{-r}/r^2.
\]
By the definition of the bound period we have, for $j > 0$, using Lemma \ref{bound-dist}, if $\xi_{\nu}(a) \in U_i'$, 
\[
|Df^j(\xi_{\nu}(a))| e^{-r} \sim |\xi_{j,i}(a) - \xi_{\nu+j}(a)| \leq e^{-\al j} \dist(\xi_{j,i}(a),Crit_a).
  \]
So we get
  \[
|\xi_{\nu+j}(a) - \xi_{\nu+j}(b)| \leq CQ^j  \frac{e^{-\al j} \dist(\xi_{j,i}(a),Crit_a)}{r^2},
\]
and therefore, since $\dist(\xi_{j,i}(a),Crit_a)$ is virtually the same for all $a\in \om$,
\[
\sum_{j=0}^{p} \frac{|\xi_{\nu + j}(a) - \xi_{\nu +j}(b)|}{\dist(\xi_{\nu+j}(b),Crit_b)} \leq \frac{C}{r^2} + C \sum_{j=1}^p  \frac{Q^j e^{-\al j}}{r^2} \leq \frac{2C}{r^2}, 
\]
where the term $C/r^2$ corresponds to $j=0$, and $\log Q < \al$. 

Between each adjacent pair of free returns there is a growth of the interval $\xi_{n,j}(\om)$ as follows. Lemma \ref{double} implies 
\begin{equation} \label{double-return}
2 \diam(\xi_{\nu_k}(\om)) \leq \diam(\xi_{\nu_{k+1}}(\om)),  \quad \text{for all $a,b \in \om$.}
\end{equation}
Let $(r)$ be those indices $k$ for which $\dist(\xi_{\nu_k}(\om), Crit_{\om}) \sim_{\sqrt{e}} e^{-r}$, and let $\hat{k}(r)$ be the largest integer in $(r)$. Hence going backwards in time, inside each $(r)$, the contribution from the bound periods is a constant times the last contribution, i.e.
\[
\sum_{k \in (r)} \frac{|\xi_{\nu_k}(a) - \xi_{\nu_k}(b)|}{\dist(\xi_{\nu_k}(b),Crit_b)} \leq  C \frac{|\xi_{\nu_{\hat{k}(r)}}(a) - \xi_{\nu_{\hat{k}(r)}(b)}|}  {\dist(\xi_{\nu_{\hat{k}(r)}}(b),Crit_b)} \leq \frac{C}{r^2}. 
  \]
Summing over all such possible returns we get 
\[
\sum_{r =\De}^{\infty} \frac{C}{r^2} \leq \frac{2C}{\De}.
  \]

Let us now look for the contribution from the free periods. Let us assume that $\nu_k$ are the returns up until $\nu_s = \nu'$, and $p_k$ their bound periods. By Lemma \ref{oel} we get that, for every $a,b \in \om$, now assuming that $\xi_j(\om) \cap U = \emptyset$ for all $\nu_k+p_k +1 \leq j \leq \nu_{k+1}-1$,
\[
  |\xi_{\nu_{k+1}-1}(a) - \xi_{\nu_{k+1}-1}(b)| \geq C'\la^{\nu_{k+1}-1-j} | \xi_{j}(a) - \xi_{j}(b)|.
  \]
Hence,
  \begin{align}
    \sum_{j=\nu_{k-1}+p_{k-1}+1}^{\nu_{k}-1} \frac{|\xi_j(a) - \xi_j(b)|}{\dist(\xi_j(b),Crit_b)} &\leq C' \sum_{j} \la^{-(\nu_{k}-1-j)} \frac{|\xi_{\nu_{k}-1}(a) - \xi_{\nu_{k}-1}(b)|}{\de} \nonumber \\
    &\leq C \frac{|\xi_{\nu_{k}-1}(a) - \xi_{\nu_{k}-1}(b)|}{\de}.
  \end{align}

  We have, for some $\ka_2 \geq 1$,
  \[
|\xi_{\nu_{k}-1}(a) - \xi_{\nu_{k}-1}(b)| \sim_{\ka_2} |\xi_{\nu_{k}}(a) - \xi_{\nu_{k}}(b)| \sim_2 e^{-r_{k}}/r_{k}^2,
\]
if $k < s$, where we have put $\dist(\xi_{\nu_k}(\om),Crit_{\om}) \sim_{\sqrt{e}} e^{-r_k}$. So for those returns the contribution to the sum (\ref{sum}) is going to be very small. 
  Recalling that $|\xi_j(a)) - \xi_j(b)| \leq S$, where $S =\vep_1 \de$ is the large scale, $\de=e^{-\De}$, we get, for the last return, 
  \begin{equation}
   \sum_{j=\nu_{s-1}+p_{s-1}+1}^{\nu_{s}-1} \frac{|\xi_j(a) - \xi_j(b)|}{\dist(\xi_j(b),Crit_b)}  \leq C \frac{S}{\de} \leq C\vep_1 ,
\end{equation}
where $C$ depends only on $C'$ and $\la$ (hence not on $\de$). So $C \vep_1$ can be made arbitrarily small if $\vep_1$ is small enough. 
We let $(r)$ be those indices $k$ such that $\dist(\xi_{\nu_k}(\om), Crit_{\om}) \sim_{\sqrt{e}} e^{-r}$, and $\hat{k}(r)$ the maximum index $k$ for which this happens. 
Then using Lemma \ref{double}, we have (\ref{double-return}), and therefore we conclude that
\[
\sum_{k \in (r)} |\xi_{\nu_{k}}(a) - \xi_{\nu_{k}}(b)| \leq C|\xi_{\nu_{\hat{k}(r)}}(a) - \xi_{\nu_{\hat{k}(r)}}(b)|.
  \]
Summing up, we get, excluding the last return,
\begin{align}
  \sum_{k=1}^{s-1} \sum_{j=\nu_{k-1}+p_{k-1} +1}^{\nu_{k}-1}\frac{|\xi_j(a) - \xi_j(b)|}{\dist(\xi_j(b),Crit_b)} &=
  \sum_{r \geq \De}  \sum_{k \in (r)} \sum_{j=\nu_{k-1}+p_{k-1}+1}^{\nu_{k}-1} \frac{|\xi_j(a) - \xi_j(b)|}{\dist(\xi_j(b),Crit_b)} \nonumber \\
  &\leq C \sum_{r \geq \De}  \sum_{k \in (r)} \frac{|\xi_{\nu_{k}-1}(a) - \xi_{\nu_{k}-1}(b)|}{\de} \nonumber \\
  &\leq C \sum_{r \geq \De}   \frac{|\xi_{\nu_{\hat{k}(r)}-1}(a) - \xi_{\nu_{\hat{k}(r)}-1}(b)|}{\de} \nonumber \\
  &\leq C \sum_{r \geq \De} \frac{e^{\De-r}}{r^2} \leq \frac{C}{\De}.
  \end{align}
  Including the last we return we get
  \begin{equation}
    \sum_{k=1}^{s} \sum_{j=\nu_{k-1}+p_{k-1} + 1}^{\nu_{k}-1}\frac{|\xi_j(a) - \xi_j(b)|}{\dist(\xi_j(b),Crit_b)} 
    \leq \frac{C}{\De} +  C\vep_1.  
    \end{equation}

    If we now pick some $n$ such that  $\nu+p \leq n < \nu'$, then letting $q_1 <  \ldots < q_t $ be consecutive, so called {\em pseudo-returns} into some fixed $U' \sm U$ so that $\nu+p \leq q_1$, $q_t \leq n$, we proceed as follows. The only difference to returns into $U$ is that we can only say that $\diam(\xi_{q_j}(\om)) \leq S$ for pseudo-returns. We do not count bound returns as pseudo-returns but consider only the free pseudo-returns. 

    The contribution to the sum (\ref{sum}) between each pair of pseudo returns is  again a constant times the last term for each pseudo-return. Let $(r)$ be the indices $l$ for which $\xi_{q_l}(\om)$ is a pseudo return for which $\dist(\xi_{q_l}(\om),Crit_{\om}) \sim_{\sqrt{e}} e^{-r}$, and let $\hat{l}(r)$ be the largest index $l$ for which $\dist(\xi_{q_l}(\om), Crit_{\om}) \sim_{\sqrt{e}} e^{-r}$.

    Then
    \begin{align}
      \sum_{j=\nu+p+1}^{q_t} \frac{|\xi_j(a) - \xi_j(b)|}{\dist(\xi_j(b),Crit_b)} &=  \sum_{j=\nu+p+1}^{q_1} \frac{|\xi_j(a) - \xi_j(b)|}{\dist(\xi_j(b),Crit_b)} \nonumber  \\
      &+ 
      \sum_{r=\De'}^{\De} \sum_{l \in (r), l > 1} \sum_{j=q_{l-1}+1}^{q_l} 
        \frac{|\xi_j(a) - \xi_j(b)|}{\dist(\xi_j(b),Crit_b)} \nonumber \\
      &\leq   C \sum_{r=\De'}^{\De} \sum_{l \in (r)} \frac{|\xi_{q_l}(a)) - \xi_{q_l}(b)|}{\dist(\xi_{q_l}(b),Crit_b)} \nonumber \\
      &\leq C \sum_{r=\De'}^{\De}  \frac{|\xi_{q_{\hat{l}(r)}}(a) - \xi_{q_{\hat{l}(r)}}(b)|}{\dist(\xi_{q_{\hat{l}(r)}}(b),Crit_b)}.
      \end{align}

Moreover, we have the assumption that $\diam(\xi_k(\om)) \leq S = \vep_1 \de$, for all $k \leq n$.  If $\xi_{q_l}(\om)$ is a pseudo return with $\dist(\xi_{q_l}(\om),Crit_{\om}) \sim_{\sqrt{e}} e^{-r_l}$, for  $\De' \leq r_l \leq \De$, the contribution will be simply bounded by $\vep_1 e^{-\De}/e^{-r_l}$.  We get 
\begin{equation}
C \sum_{r=\De'}^{\De}  \frac{|\xi_{q_{\hat{l}(r)}}(a) - \xi_{q_{\hat{l}(r)}}(b)|}{\dist(\xi_{q_{\hat{l}(r)}}(b),Crit_b)}
  \leq C \sum_{r=\De'}^{\De} \vep_1 e^{r-\De}  \leq C \vep_1. 
\end{equation}
The contribution from the very last iterates from $q_t < j \leq n$ is a constant (depending on the large scale) by the uniform expansion along the early orbit (the bound period) and then outside $U'$. Summing up,
\[
 \sum_{j=1}^{n} \frac{|\xi_j(a) - \xi_j(b)|}{\dist(\xi_j(b),Crit_b)} \leq \frac{2C}{\De} + \frac{C}{\De} + 2C\vep_1,
  \]
which can be made small if $\vep_1$ and $\de$ are small enough. This finishes the lemma. 
\end{proof}

We now get {\em a posteriori} that $\xi_n$ is almost affine on each partition element $\om$. Hence also $|\xi_n(a) - \xi_n(b)|$ expands according to the space derivative for any parameter $c \in [a,b]$ i.e.
\[
|\xi_n(a) - \xi_n(b)| \sim_C |Df_c^{n-j}(\xi_j(c))| |\xi_j(a) - \xi_j(b)|,
\]
for $C > 1$ close to $1$. This is called {\em strong distortion}. 

Moreover, we see that as long as $\ga \geq \ga_I$ for returns (in general $\ga \geq \ga_I - 4 \al K$), we have good geometry control, i.e. for a partition element $\om$, we have (\ref{good-geometry}), for all $a,b \in \om$.

\subsection{Initial distortion}

As a direct consequence of the Main distortion lemma, we here state that for any sufficiently small $\vep$ we can find an interval $\om \subset (-\vep,\vep)$ such that $\xi_{n,l}(\om)$ grows to some ``large scale'' (denoted by $S$) or returns into $U$ as an essential first return. 

\begin{Lem}[Start-lemma] \label{startlemma}
Let $f=f_0$ be as in Theorem \ref{sats1} and let $\vep' > 0$ and $N >0 $ from Lemma \ref{transv1}. There is a neighbourhood $U$ of $Crit_0$ and a number $S > 0$ (called the ``large scale''), which depends on $U$ such that the following holds. For every sufficiently small $\vep > 0$ and each critical point $c_l$ there is some $N_l \geq N > 0$ such that for every $a \in \om=(-\vep,\vep)$ we have:
\begin{itemize}
\item[i)]
For some $\ga_l \geq \ga_0 (1-\vep')$, it holds that
\[
|Df_a^k(f_a(c_l(a)))| \geq Ce^{\ga_l k},  \text{ for all $k \leq N_l$},
\]
\item[ii)]
for all $k \leq N_l - 1$, it holds that 
\[
\diam(\xi_{k,l}(\om)) \leq \left\{
\begin{array}{cc}
\frac{\dist(\xi_{k,l}(\om),Jrit_{\om})}{(\log (\dist(\xi_{k,l}(\om),Jrit_{\om})))^2} , & \text{if } \xi_{k,l}(\om) \cap U \neq \emptyset, \nonumber  \\
S, & \text{if } \xi_{k,l}(\om) \cap U = \emptyset, \nonumber
\end{array} \right.
\]
\item[iii)]
for $k = N_l$, it holds that 
\[
\diam(\xi_{N_l,l}(\om)) \geq \left\{
\begin{array}{cc}
\frac{\dist(\xi_{N_l,l}(\om),Jrit_{\om})}{(\log (\dist(\xi_{N_l,l}(\om),Jrit_{\om})))^2} , & \text{if } \xi_{N_l,l}(\om) \cap U \neq \emptyset,  \nonumber \\
S, & \text{if } \xi_{N_l,l}(\om) \cap U = \emptyset, \nonumber
\end{array} \right.
\]
\item[iv)]
and finally, for all $a,b \in \om$ it holds that
\[
\biggl| \frac{Df_a^{n-N}(\xi_{N}(a))}{Df_b^{n-N}(\xi_{N}(b))}  - 1 \biggr|  \leq \vep', \text{ for all $n \leq N_l$}. 
\]
\end{itemize}
\end{Lem}

\begin{Rem}
The Whitney type of condition on the diameter of $\xi_n(\om)$ and its distance to the critical points has the following meaning. With $\dist(\xi_{n,l}(\om),Jrit_{\om}) \sim e^{-r}$  the diameter becomes $\sim e^{-r}/r^2$ and this is sufficient for having control of the distortion of the derivative. The condition is also used in the main distortion lemma later. 
\end{Rem}

\begin{proof}
  Similar to the proof of Lemma \ref{weak-distortion},
we can easily conclude that condition $ii)$  implies that we have very small distortion for $Df_c$ $c \in (-\vep,\vep)$ both in the space and parameter variable. 
We see that if $\xi_n(a)$ and $\xi_n(b)$ are close in this sense, then for all $a,b \in (-\vep,\vep)$, in particular for $b=0$, the bounded distortion on $Df_c$ implies
\[
|Df_a^n(v_l(a))| \geq C^{-n} |Df_0^n(v_l(0))| \geq C_0 e^{\ga_1 n},
\]
for some $\ga_1$ slightly smaller than $\ga_0$ (we may assume that $\ga_1 \geq (1-\vep') \ga_0$), if $C$ is close enough to $1$ (by choosing $S$ small enough). So $f_a$ also satisfies the CE-condition with exponent slightly smaller than $\ga_0$. We now let $N_l$ be the maximal integer such that $ii)$ holds. We have shown that also $i)$ holds for this $N_l$. We can hence use the Main distortion lemma for all following returns until time $N_l$. Hence $iv)$ holds. 
\end{proof}


\subsection{The partition}
If $J(f) = \hat{\C}$ then we have $2d-2$ critical points, counting multiplicity. Inside the subspace $\La_{d,\oli{p}'}$ each critical point moves  analytically. So Lemma \ref{startlemma} gives at most $2d-2$ numbers $N_l$, given an interval $\om_0=(-\vep,\vep)$, such that $\xi_{N_l,l}(\om_0)$ has grown to some large scale $S$ (same for all $l$), or has reached size $e^{-r}/r^2$ inside $U$, where $e^{-r}$ is, more or less, the distance to the critical points, i.e. $\dist(\xi_{N_l,l}(\om_0),Jrit_{\om_0}) \sim e^{-r}$. We now assume that, without loss of generality, $N_1 = \min (N_l)$. Thus we have the CE-condition satisfied for all critical points up until time $N_1$, on $\om_0$. 

If $N_1$ is not a return time, we have $\diam(\xi_{N_1,1}(\om_0)) \geq S$ by Lemma \ref{startlemma}. As soon as this happens, we partition the interval $\om_0$ into the least number of smaller sub-intervals $\om_0^i \subset \om$ of equal length such that $\diam(\xi_{N_1,1}(\om_0^i)) \leq S$. We call the sets $\om_0^i$ of this type {\em partition elements}. We do this partitioning for every critical point at all times outside $U$ until some parameter returns into $U$. In this way we always have $\diam(\xi_{n,l}(\om)) \leq S$ for {\em any} partition element $\om$ and study the evolution of each such $\om$ separately. We will use $\om \subset \om_0 = (-\vep,\vep)$ as a standard notion for partition elements in the future.

Let us go back to the critical point $c_1$ ($l=1$)
and assume that $\om \subset \om_0$ is such partition element and that $m_1$ is the smallest integer $m_1 \geq N_1$ such that $\xi_{m_{1},1}(\om) \cap U \neq \emptyset$, i.e. $\xi_{m_1,1}(\om)$ is a {\em return} into $U$. If
\[
\frac{1}{2} \frac{\dist(\xi_{m_1,1}(\om),Jrit_{\om})}{(\log (\dist(\xi_{m_1,l}(\om),Jrit_{\om})))^2} \leq \diam(\xi_{m_1,1}(\om)) ,
\]
we speak of an {\em essential} return. Otherwise the return is {\em inessential}. For essential returns we then partition the interval $\om$ into smaller intervals $\om_{m_1}^i \subset \om$ such that
\begin{multline} \label{dyadic2}
  \frac{1}{2} \frac{\dist(\xi_{m_1,1}(\om_{m_1}^i),Jrit_{\om_{m_1}^i})}{(\log (\dist(\xi_{m_1,l}(\om_{m_1}^i),Jrit_{\om_{m_1}^i})))^2} \\
  \leq \diam(\xi_{m_1,1}(\om_{m_1}^i))  \\ \leq \frac{\dist(\xi_{m_1,1}(\om_{m_1}^i),Jrit_{\om_{m_1}^i})}{(\log (\dist(\xi_{m_1,l}(\om_{m_1}^i),Jrit_{\om_{m_1}^i})))^2}.
\end{multline}
These smaller intervals $\om_{m_1}^i$ are also called {\em partition elements} (at time $m_1$). 
The condition (\ref{dyadic2}) implies that we have control of the distortion:
\[
\frac{|Df_a(\xi_{m_1,1}(a))|}{|Df_b(\xi_{m_1,1}(b))|} \leq C(\ti{r}), \quad \text{ for all $a,b \in \om_{m_1}^i$}.
\]
where $C(\ti{r}) > 1$ and tends to $1$ as $\ti{r} = -\log (\dist(\xi_{m_1,1}(\om_{m_1}^i),Jrit_{\om_{m_1}^i}))$ tends to infinity. We let $r = \lceil \ti{r} - 1/2 \rceil$. For $\om_{m_1}^i$ above, we associate $r = r(\ti{r})$ to $\ti{r}$, and it follows that
\[
\dist(\xi_{m_1,1}(\om_{m_1}^i),Jrit_{\om_{m_1}^i}) \sim_{\sqrt{e}} e^{-r}, 
\]
(cf. with the annular neighbourhoods in \cite{MA-Z}). Moreover, we see that
\[
  \diam(\xi_{m_1,1}(\om_{m_1}^i)) \sim_2 e^{-r}/r^2,
\]
if $r \geq \De$, and $\De$ sufficiently large. When we write $\dist(A,B) \sim_{\sqrt{e}} e^{-r}$, typically we use it when $A=\xi_{n,l}(\om)$ and $B=Crit_{\om}$, then we mean the unique $r$ such that $r = \lceil -\log \dist(A,B) -1/2 \rceil$, i.e. $\dist(A,B) \in [e^{-r-1/2},e^{-r+1/2})$.


For each return, and in particular this first return, we partition parameter intervals according to the above rule. Moreover, we delete parameters not satisfying the basic assumption and show later that the Lebesgue measure of the set deleted is a small portion of the total interval returning into $U$. It is quite easy to see that this is the case for the first return. Because of the slow recurrence condition, we see that
\[
e^{-r} \geq e^{-\al m_1} \gg e^{-2\al m_1}. 
\]
Hence, the basic assumption possibly forces us to delete a small fraction of parameters at time $m_1$.

After this return the first bound period starts, and the whole idea is that binding the old orbit to the early orbit of possibly another critical point, will, via distortion control, transfer the derivative gain form the early orbit to the old orbit. To do this we need to be able to use the binding time for all critical points in the induction. We continue like this as long as we can use the binding information for all critical points, up until time $N_1$. This procedure creates a Cantor-like set (denoted by $\Om_l(m)$) of ``good'' parameters, for each critical point $c_l$, that do satisfy the basic assumption up until some time $m$, which turns out to be much larger than $N_1$, because the bound periods for a return $\xi_{m,l}(\om)$ into $U$ are much smaller than $m$ itself (Lemma \ref{boundexp}).

At this point, we have to delete more parameters such that the binding period can be used longer. A potential problem here is that different critical points $c_l$ may produce different Cantor-like sets up until time $m$, and if we take intersections of these sets, we may destroy the partition elements. But the idea is that the partition elements at time from, say $N_1$ until $2N_1$, are much larger that those partition elements formed around time $m \gg N_1$. We develop this idea, which is due to M. Benedicks, later.

In the construction, the growth of the derivative along critical orbits is never allowed to go below a certain level, in order to have the whole machine working. Recall that $\ga_B = (3/4) \min (\ga_0,\ga_H) (1-\tau)$, where $0 < \tau  < 1$. This exponent $\ga_B$ should be thought of the desired Lyapunov exponent, which we will get at the end. It will also be used as an induction assumption. The number $\tau$ can be chosen freely but $\de$ depends on it (see Section \ref{large-deviations}, Lemma \ref{esctimes-estimate}). The intermediate Lyapunov exponent $\ga_I = (1/3) \min (\ga_0,\ga_H) (1-\tau) < \ga_B/2$ will be an assumption in most lemmas.   



\section{Large deviations} \label{large-deviations}

We will make an induction over time intervals of the type $[n,2n]$ and assume from now on that we make partitions as described above. Given a good situation at time $n$ with growth of the derivative, we first delete the parameters not satisfying the basic assumption up until time $2n$. But according to Lemma \ref{adjacent-returns}, this means that we may lose some part of the Lyapunov exponent. Therefore we make use of the famous large deviation argument, developed by Benedicks and Carleson, to restore the Lyapunov exponent up until time $2n$. 

This section is very similar to older papers \cite{BC2}, \cite{MA-Z} et al.  

\begin{Lem} \label{ba-measure}
  Suppose that $\xi_{\nu,l}(\om)$ is an essential return into $U_i$, and that the Lyapunov exponent $\ga \geq \ga_I$ for all critical points, $\om \subset \EE_{\nu,l}(\ga) \cap \BB_{\nu,l}$. Then if $\nu'$ is the next return time, we have that the set $\hat{\om}$ of parameters in $\om$ that satisfies the basic assumption, has Lebesgue measure
  \[
m(\hat{\om}) \geq (1-e^{-\al \nu}) m(\om). 
    \]
  \end{Lem}
  \begin{proof}
    This follows quite easily, since the interval $\xi_{\nu+p}(\om)$ grows rapidly during the bound period $p$. By Lemma \ref{boundexp}, Lemma \ref{bound-dist}, Lemma \ref{oel}, and Lemma \ref{main-distortion}, we get, for any $a\in \om$,
    \begin{align}
      \diam(\xi_{\nu+p+1,l}(\om)) &\sim \frac{e^{-rd_i}}{r^2} |Df^p(\xi_{\nu+1,l}(a))| \nonumber \\
      &\sim  \frac{|\xi_{\nu+p+1,l}(a) - \xi_{p+1,i}(a) |}{r^2} \nonumber \\
      &\geq C e^{-\al (p+1) - 2 \log r} \dist(\xi_{p+1,i}(a),Crit_{a}) \nonumber \\
      &\geq C K_b e^{-2\al (p + 1) - \al (p+1) - 2 \log r} \geq e^{-(7/2) \al p - 2 \log r},
    \end{align}
    if $p$ is large. So,
    \begin{multline} \nonumber 
      \diam(\xi_{\nu',l}(\om)) \geq \diam(\xi_{\nu+p,l}(\om)) C' \la_1^{\nu'-(\nu+p)} \geq C' e^{-(7/2) \al p - 2\log r} \\
      \geq  e^{-7 \al d r/\ga - 2\log r } \geq e^{ - \frac{8 \al d}{\ga} r}.  
    \end{multline}
Recalling the distortion control from the main distortion Lemma \ref{main-distortion} together with Lemma \ref{transv1}, we see that the measure of parameters deleted at $\nu'$ is 
    \[
     \frac{ |\om| - |\hat{\om}|}{|\om|}  \leq 2 \frac{e^{-2 \al \nu'}}{\diam(\xi_{\nu'}(\om))} \leq
     2e^{-\al ( 2 - \frac{8 \al d}{\ga} ) \nu} \leq e^{-\al \nu},
        \]
        since $\al K/\ga \leq 1/100$, $K = \max(d)$ (maximal degree of the critical points).  
    \end{proof}


      We now define escape time and escape situation. Let $U^2$ be a neighbourhood of $Crit(f_0)$ such that $U^2 = \cup_j B(c_j, \de^2) \subset U$.  We say that a {\em deep return} is characterised by $\xi_{n,l}(\om) \cap U^2 \neq \emptyset$ and a {\em shallow return} means that $\xi_{n,l}(\om) \cap U^2 = \emptyset$ but $\xi_{n,l}(\om) \cap U \neq \emptyset$. We then speak of deep returns into $U^2$
      and shallow returns into $U \sm U^2$ even if the actual curve $\xi_{n,l}(\om)$ does not entirely lay inside $U^2$ or $U \sm U^2$ respectively. 
      We also let $\om_n(a)$ be the corresponding partition element following the parameter $a$, i.e. the unique $\om$ such that $\xi_n(\om)$ has diameter bounded by $S$ if $ \xi_n(\om) \cap U = \emptyset$ and bounded by $\dist(\xi_n(\om),Crit_{\om}) / (\log \dist(\xi_n(\om),Crit_{\om}) )^2$ if $ \xi_n(\om) \cap U \neq \emptyset$. 
      \begin{Def}
        We say that $\xi_{n}(\om)$ has escaped, or is in escape position, if $\diam(\xi_n(\om)) \geq S$, and the bound period has passed. 

        The escape time for a parameter $a \in \om$ for a deep return $\xi_{\nu,l}(\om)$ into $U^2$ is defined as the least number $n - (\nu + p) \geq 0$ (where $p$ is the bound period for the return) such that $\xi_{n,l}(\om_n(a))$ has reached escape position. We write $E_l(a,\nu) = n-(\nu + p)$ for this escape time. We also define the escape time for shallow returns, i.e. if  $\xi_{\nu,l}(\om) \subset U \sm U^2$, to be equal to zero.

        If some parameter $a \in \om$ has that $\xi_{\nu',l}(\om_{\nu'}(a))$ does not satisfy the basic approach rate condition, i.e. returns too deep for some $\nu' > \nu$ before it escapes, then those parameters get deleted and we put $E_l(a,\nu) = -\infty$.   
        \end{Def}
      
    \begin{Lem} \label{q-time}
Suppose that $\xi_{\nu,l}(\om)$ is an essential return into $U_i$, $\om \in \EE_{\nu,l}(\ga) \cap \BB_{\nu,l}$, $\ga \geq \ga_I$ and that $\dist(\xi_{\nu,l}(\om),Crit_{\om}) \sim_{\sqrt{e}} e^{-r}$. Put $h = 4K^2/\ga_I$. Then if $q = n - (\nu + p)$ where $n$ is the next essential return or the time when $\xi_{n,l}(\om)$ is in escape position, which ever comes first, we have the estimate
      \[
q \leq hr.
\]
      \end{Lem}
\begin{proof}
Let us put $D_j = |Df^j(\xi_{\nu,l}(a))|$, for $a \in \om$.  By the definition of the bound period, the basic assumption, and Lemma \ref{boundexp}, for all $a \in \om$, 
  \begin{multline} \label{b-exp}
    D_{p+1} \geq C e^{-\al (p+1)} \dist(\xi_{p+1,i}(a), Crit_{a}) e^r \\
    \geq e^{-2 \al(p+1) - \al (p+1)  + r}  \geq e^{r(1-\frac{7 \al K}{\ga})} \geq e^{r(1-\frac{7 \al K}{\ga_I})}.
\end{multline}
Let $m_j$ be the inessential returns after $\nu$, i.e. $\nu < m_1 < m_2 < \ldots < m_s < n$. Let $p_j$ and $q_j$ be the bound and free periods respectively following $m_j$. Let $p_0$ and $q_0$ be the bound and free periods following the return $\nu$. It can happen that escape takes place before a return takes place, and then $q_s$ is not a complete free period. It can also happen that $n$ is a time during the bound period for $m_s$. But then we have $m_s+p_s - \nu$ as an upper bound for $q$ and we can assume that $q > m_s+p_s$.

Suppose that $\dist(\xi_{m_j}(\om),Crit_{\om}) \sim_{\sqrt{e}} e^{-r_j}$, and let $r=r_0$. Suppose that $n=\nu'$ is a return. Then, as long as the bound period is bounded by $(2 K \al /\ga_I)\nu$, we can use the same estimate as (\ref{b-exp}), and Lemma \ref{oel}, to obtain
\begin{align}
  \diam(&\xi_{n}(\om)) \sim |Df_a^{n - \nu}(\xi_{\nu}(a))| \diam(\xi_{\nu}(\om)) \nonumber \\
  &=
    \prod_{j=0}^{s} |Df_a^{p_j}(\xi_{m_j}(a))| C' e^{\ga_H q_j} \diam(\xi_{\nu}(\om)) \nonumber \\
  &\geq
    e^{r(1-\frac{7 \al K}{\ga_I})}  C' e^{\ga_H q_0} \diam(\xi_{\nu}(\om)) \prod_{j=1}^s e^{r_j(1-\frac{7 \al K}{\ga_I})} \prod_{j=1}^s C' e^{q_j \ga_H}   \nonumber \\
  &\geq
 e^{-r \frac{8 \al K}{\ga_I} + q_0 \ga_H} \prod_{j=1}^s e^{r_j(1-\frac{8 \al K}{\ga_I})  +  q_j \ga_H}.  \label{last}
\end{align}

If $n$ was not a return, then let $q_1 < \ldots < q_t$ be the pseudo-returns after $m_s+p_s$. Between each pair of pseudo-returns we have uniform expansion of the derivative according to Lemma \ref{adjacent-returns}.  Between $m_s+p_s$ and $q_1$ we also have uniform expansion according to Lemma \ref{oel}. So we only need to conider the last time period, from $q_t$ to $n$. Since $\xi_{q_t}(\om)$ may belong to $U' \sm U$ we have $|Df_a(\xi_{q_t}(a))| \geq e^{-K \De}$ for all $a \in \om$. After time $q_t$ we can use the binding information, Lemma \ref{bound-dist} and the first statement of Lemma \ref{oel} with $U = U'$, depending on whether $n$ belongs to the bound period or not. In any case we get uniform expansion; $|Df_a^{n-q_t-1}(\xi_{q_t+1}(a))| \geq C e^{\min (\ga,\ga_H) (n-q_t-1)}$.  In other words, with $z = \xi_{m_s+p_s}(a)$, for $a \in \om$, 
\begin{align}
  |Df_a^{n-(m_s+p_s)}(z)| &= |Df_a^{q_1-(m_s+p_s)}(z)| |Df_a^{q_2-q_1}(f_a^{q_1-(m_s+p_s)}(z))| \nonumber \\
    &\cdot \ldots \cdot |Df_a^{q_t-q_{t-1}}(f_a^{q_{t-1}-(m_s+p_s)}(z))| |Df_a(f_a^{q_t-(m_s+p_s)}(z))| \nonumber \\
  &\cdot |Df_a^{n-q_t-1}(f_a^{q_t-(m_s+p_s) + 1}(z))| \nonumber \\
                          &\geq C' e^{\ga_H (q_1 -(m_s+p_s))} e^{\ga_1 (q_t-q_1)} e^{-K \De} C e^{\min(\ga,\ga_H) (n-q_t-1)} \nonumber \\
                            &\geq e^{\ga_C (n-(m_s+p_s))} e^{-K\De},
\end{align}
since $\ga_1 \geq (9/10) \ga_I \geq \ga_C$. 
So we may have to replace $q_s \ga_H$ with $\ga_C q_s - K \De$ in (\ref{last}), where $q_s = (n-(m_s+p_s))$ in this case.

Since $\ga_C < \ga_H$, and $\diam(\xi_{n}(\om))$ is assumed to be at most $S = \vep_1 \de \leq 1$, we therefore get
\[
\sum_{j=1}^s  r_j (1-\frac{8 \al K}{\ga_I}) + \sum_{j=0}^s q_j \ga_C \leq r \frac{8 \al K}{\ga_I} + K\De. 
\]
Hence, if $q  = \sum_{j=1}^s p_j + \sum_{j=0}^s q_j$, we get
\begin{align}
  q &= \sum_{j=1}^s p_j + \sum_{j=0}^s q_j \leq \sum_{j=1}^s \frac{2K}{\ga_I} r_j + \sum_{j=0}^s q_j \nonumber \\
    &= \sum_{j=1}^s \frac{4 K}{\ga_I} (1 - \frac{8 \al K}{\ga_I}) r_j + \frac{1}{\ga_C} \sum_{j=0}^s q_j \ga_C \nonumber \\
  &\leq \max \biggl( \frac{4 K}{\ga_I}, \frac{1}{\ga_C} \biggr)
  \biggl(  \sum_{j=1}^s  r_j (1-\frac{8 \al K}{\ga_I}) + \sum_{j=0}^s q_j\ga_C \biggr)  \nonumber \\
  &\leq \max \biggl( \frac{4 K}{\ga_I}, \frac{1}{\ga_C} \biggr) \biggl( \frac{8 \al K}{\ga_I} r +K\De \biggr) \leq \frac{4K^2}{\ga_I}r,
\end{align}
since $4K/\ga_I < 1/\ga_C$, and where we also used that $\al \leq 3 \ga_I /(400K \Ga)$. Since $Kr/\ga_I \leq (2 K \al /\ga_I)\nu$, we can use the binding information the whole time.
\end{proof}

We will now estimate the measure of the set of parameters having a specific history for the returns in a time window of the form $[n,2n]$. For simplicity, suppose that $\xi_{\nu}(\om_0)$ is an essential return with $\dist(\xi_{\nu}(\om_0),Crit_{\om_0}) \sim_{\sqrt{e}} e^{-r_0}$ and $\nu \geq n$ ($\nu$ should be though of as the smallest return time after $n$).
Let us study the evolution of $\xi_m(\om_m(a))$ as $m$ goes through a sequence of essential returns $\nu_1, \nu_2, \ldots, \nu_s \leq 2n$. Let us also assume that $\om_{\nu_j}(a) \subset \EE_{\nu_j,l}(\ga_I) \cap \BB_{\nu_j,l}$, for these returns so that we can use the binding information of all other critical points up to time $2n$. This is not a strong assumption, as we now explain. Suppose $a \in \EE_{n,l}(\ga_B) \cap \BB_{2n,l}$, i.e. we assume that the basic approach rate condition is fulfilled up until time $2n$. The Lyapunov exponent will not drop too much at each return in the interval $[n,2n]$, because we can use Lemma \ref{boundexp} at each return and get a trivial lower bound for the expansion, namely $1$ during the bound period. But this means that the actual Lyapunov exponent is bounded from below, and we get a trivial bound,
\[
|Df^{2n}(v_l(a))| \geq e^{\ga_B n} \geq e^{2n \ga_I}, 
  \]
since $\ga_I \leq \ga_B/2$. In other words, $a \in \EE_{2n,l}(\ga_I) \cap \BB_{2n,l}$. 

By the Main Distortion Lemma \ref{main-distortion}, which then gives good geometry control, the diameter of $\xi_{\nu_j+p_j}(\om_{\nu_j+p_j}(a))$ is more or less equal to the length of the curve (which is then more or less straight), i.e. $\sim e^{-(7 K \al/\ga) r_j}$, see inequality (\ref{b-exp}). After the free period it may expand further, and to get rid of the constant $C'$ in Lemma \ref{oel}, we may say that the curve
$\xi_{\nu_{j+1}}(\om_{\nu_{j+1}}(a))$ has a diameter at least $e^{-(8 K \al / \ga) r_j}$. 
We therefore get, with $\ga \geq \ga_I$, that the measure of those parameters $b \in \om_{\nu_j}(a)$ entering into $U$ with  $\dist(\xi_{\nu_{j+1}}(b), Crit_{b} ) \sim_{\sqrt{e}} e^{-r_{j+1}}$ is
\begin{equation}
 m(\om_{\nu_{j+1}}(a)) =  m( \{ b \in \om_{\nu_{j}}(a) : \xi_{\nu_{j+1}}(b) \sim_{\sqrt{e}} e^{-r_{j+1}} \} )
  \leq C \frac{e^{-r_{j+1}}}{e^{-(8 K \al/\ga) r_j}}  m(\om_{\nu_j(a)}) ,
\end{equation}
(recall that we do not partition $\om_{\nu_j}(a)$ until the next return, so $\om_{\nu_j}(a)= \om_{\nu_{j+1}-1}(a)$). 
So suppose now that we have a sequence of $s$ essential returns $\nu_1, \nu_2, \ldots, \nu_s \leq 2n$. Let us also assume that we always have a lower bound, $\ga_I$, for the Lyapunov exponent, i.e. $a \in \EE_{2n,l}(\ga_I) \cup \BB_{2n,l}$ for the parameters we are considering. Then the portion from the starting interval, call it $\om_0 = \om_{\nu}(a)$ for some $a \in \om_0$, that has this specific history is, with $\om_j = \om_{\nu_j}(a)$, 
\begin{equation} \label{portion}
  \frac{m(\om_s)}{m(\om_0)} = \prod_{j=0}^{s-1} \frac{m(\om_{j+1})}{m(\om_j)} \leq C^s \prod_{j=0}^{s-1} \frac{e^{-r_{j+1}}}{e^{-(8 K \al/\ga) r_j}}.
\end{equation}

We continue to follow \cite{MA-Z} and \cite{BC2} more or less verbatim. Let $R=r_1+r_2 + \ldots + r_s$. We now compute the number of combinations of choosing such $r_j$ given that $r_j \geq \De \geq 0$. 
Let us not yet take into account that we are partitioning the intervals into smaller intervals such that
  \begin{equation} \label{squared-partitions}
\diam(\xi_{\nu_{j}}(\om)) \sim_{\sqrt{e}} e^{-r_{j}} /r_{j}^2 , \text{ for each $j = 1, \ldots, s$,}
\end{equation}
where $\om = \om_{\nu_j}(a)$. Hence for each such set we have another $r_{j}^2$ possibilities.

By the pigeonhole principle, this number of combinations is, disregarding from these extra $r_j^2s$ possibilities, 
  \[
    \binom{R+s-1}{s-1}.
  \]
  By Stirling's formula this can be estimated as follows, using that $R \geq s \De$, 
  \begin{align}
    \binom{R+s-1}{s-1} &\leq C \frac{1}{\sqrt{2\pi}} \frac{(R+s-1)^{R+s-1} e^{-R-s+1}}{R^R e^{-R} (s-1)^{s-1} e^{-s}}      \sqrt{\frac{R+s-1}{R(s-1)}} \nonumber \\
    &\leq \frac{R^{R+\frac{R}{\De}} (1+\frac{1}{\De})^{(1+\frac{1}{\De})R}}{R^R (\frac{R}{\De})^{R/\De}} \nonumber \\
    &\leq \biggl( \De^{1/\De} (1 + \frac{1}{\De})^{1+\frac{1}{\De}} \biggr)^R \leq 2(1 + \eta(\De))^R
\end{align}
if $\De$ is large enough, where $\eta(\De) = \OO(1/\De)$. 

Taking into account now (\ref{squared-partitions}), we get that the number of combinations is
\[
2(1+\eta(\De))^R \prod_{j=1}^s r_j^2 \leq e^{R/32} (1 + \eta(\De))^R.
  \]


We can rewrite (\ref{portion}) to get, (recall the condition on $\al$),
\[
 \frac{m(\om_s)}{m(\om_0)}= C^s e^{r_0 (8 \al K/\ga) - \sum_{j=1}^{s-1} r_j (1- 8 \al K/\ga) - r_s } \leq C^se^{r_0 (8 \al K/\ga) - (15/16) R }.
\]

Given an essential return $\xi_{\nu,l}(\om)$, let $A_{s,R} \subset \om$ be the set of those parameters having exactly $s$ essential returns as above before escaping at the $s+1$:st return, for a fixed $R$. Each pair of sequences $\{\nu_j \}_{j=1}^s , \{ r_j\}_{j=1}^s$ defines a unique history for a parameter $a \in A_{s,R}$. Letting $s$ and $R$ vary, then $\om$ gets partitioned into a (likely huge) number of smaller intervals having this specific history. But let us fix $s$ and let $\hat{\om}_s$ be the largest of these partition intervals for this fixed $s$. Then
\[
|A_{s,R}| \leq |\hat{\om}_s| e^{R/32}(1 + \eta(\De))^R.
\]

Now we show that the set of those parameters for which $\xi_n(a)$ returns too frequently and too deep into $U$ has very small Lebesgue measure. This is handled via so famous large deviation argument, originally developed in \cite{BC2}, which is an idea from a probabilistic point of view, and although the system we are considering is deterministic.

For an essential return $\xi_{\nu,l}(\om)$ into $U^2$ where $\dist(\xi_{\nu,l}(\om), Crit_{\om}) \sim_{\sqrt{e}} e^{-r}$, suppose that $a \in \om$ has $s$ essential returns before it has escaped. Then according to Lemma \ref{q-time} , we have 
\[
E_l(a,\nu) \leq \sum_{j=0}^s h r_j = hr + hR,
\]
where $R = r_1 + \ldots + r_s$. So the escape time $t \leq hr  + hR$, i.e. it is bounded in terms of how deep the returns are. Let us estimate the measure of those parameters that escape at a certain (long) time $t$. 

Put $r=r_0$. We get, given that $\De$ is large enough,
\begin{align}
  m(\{ a \in \om : E_l(a,\nu) = t \}) &\leq \sum_{R \geq t/h - r_0, s \leq R/\De} |A_{s,R}|  \nonumber \\
  &\leq \sum_{R \geq t/h - r_0, s \leq R/\De} |\hat{\om}_s| e^{R/32}(1 + \eta(\De))^R \nonumber \\
                                  &\leq |\om| \sum_{R=t/h - r_0}^{\infty} \sum_{s=1}^{R/\De} e^{R/32}(1 + \eta(\De))^R C^s e^{r_0 (8 \al K /\ga) - (15/16) R } \nonumber \\
                                  &\leq  C' |\om| \sum_{R=t/h - r_0}^{\infty} C^{R/\De} e^{-R(\frac{29}{32} -  \eta(\De)) + (8 K \al /\ga)r_0} \nonumber \\
                                  &\leq  C' |\om| e^{-(\frac{t}{h} - r_0)\frac{28}{32}  + (8 K \al /\ga)r_0} \nonumber \\
  &\leq C' |\om| e^{-\frac{t}{h}\frac{28}{32} + (\frac{28}{32} + \frac{8 K\al}{\ga}) r_0}.
\end{align}
for some constant $C' > 0$. 

By the condition on $\al$, if $\ga \geq \ga_I$, we get an estimate of the measure of parameters for large escape times. Let us suppose that $t > 2hr_0$. Then
\begin{equation} \label{long-escapes}
 m(\{ a \in \om : E_l(a,\nu) = t \}) \leq Ce^{-\frac{t}{3h}} |\om|. 
\end{equation}

We now follow a parameter in $a \in \om$ in a time window $[n,2n]$, and estimate its total time spent on escaping from essential returns. Recall that given an essential return $\xi_{\nu}(\om_{\nu}(a))$, the parameter $a$ has to escape first before we can start counting the next escape time. Let
\[
T_n(a) = T_{n,l}(a) = \sum_{j=1}^{s(a)} E_l(a,\nu_j(a)),
\]
where $\nu_j(a)$ are essential returns after escape situations, and $s=s(a)$ the total number of such returns in $[n,2n]$. We include shallow returns above also but then, by definition, the escape time is zero, so one needs only consider deep returns in the sum.

\begin{Rem} \label{blind-escapes}
A note on the last return $\nu_s$ in the expression of $T_{n,l}(a)$. The escape period of the last return $\nu_s=\nu_s(a)$, by definition, has to transcend into the next time window $[2n,4n]$. If it is too long it may deteriorate the Lyapunov exponent for that parameter too much. Here we make the following convention, namely that if $E(a,\nu_s) \geq 6 h \al n$ (where $6 h \al n \ll n$), then we delete those parameters. They constitute an exponentially small portion of the parameters in $\om$ (put $t=6h \al n$ in equation (\ref{long-escapes})), i.e. has measure $\leq |\om| Ce^{- q n}$, where $q = 2 \al$. We simply disregard from those parameters in the above expression for $T_n(a)$. They can easily be taken care of in the final proof in the next section. 
\end{Rem}

In order to reach the main conclusion that the set of parameters having too many too deep returns in the time window $[n,2n]$ has small measure, we want to estimate, for suitable $\th > 0$, the integral
\[
\frac{1}{|\om|} \int_{\om} e^{\th T_n(a)} \ud a.
  \]
There is some freedom of how to choose $\th$, but let us set $\th = 1/(6h)$. 
\begin{Lem} \label{escape-times}
  Let $\xi_{\nu,l}(\om)$ be a deep essential return with $\dist(\xi_{\nu,l}(\om), Crit_{\om}) \sim_{\sqrt{e}} e^{-r}$, $n \leq \nu \leq 2n$, and $\om \subset \EE_{\nu,l,\star} (\ga) \cap \BB_{\nu,l,\star}$ for some $\ga \geq \ga_I$. Suppose also that all parameters $a \in \om \cap \BB_{2n,l}$ has that $a \in \EE_{2n,l}(\ga_I)$. Then
\begin{align}
  \int_{ \{ a \in \om: 2hr \leq E_l(a,\nu) \leq \nu-n  \}  } e^{\th E_l(a,\nu)} \ud a &\leq C e^{-r/3} |\om|, \\
  \int_{  \{ a \in \om: E_l(a,\nu) \leq 2hr \} } e^{\th E_l(a,\nu)} \ud a &\leq C e^{r/3} |\om|.
\end{align}
\end{Lem}
\begin{proof}
By (\ref{long-escapes}) we have,
\begin{align}
  \int_{  \{ a \in \om : E_l(a,\nu) \geq 2hr \}  } e^{\th E_l(a,\nu)} \ud a &\leq C \sum_{t \geq 2hr} e^{-\frac{t}{3h}} e^{\th t} |\om|  \nonumber \\
  &\leq C e^{-\frac{t}{6h}} |\om| \leq C e^{-r/3}|\om|.
  \end{align}
The second inequality follows directly. 
  \end{proof}

\begin{Lem} \label{escape-measure}
  Let $\xi_{\nu,l}(\om)$ be an essential return with $\dist(\xi_{\nu,l}(\om), Crit_{\om}) \sim_{\sqrt{e}} e^{-r}$, $n \leq \nu \leq 2n$, and $\om \subset \EE_{\nu,l,\star} (\ga) \cap \BB_{\nu,l,\star}$ for some $\ga \geq \ga_I$. Suppose also that all parameters $a \in \om \cap \BB_{2n,l}$ has that $a \in \EE_{2n,l}(\ga_I)$. Then for any $\vep_2 > 0$ there is a $\De_2$ such that if $\De \geq \De_2$ (recall $\de=e^{-\De}$), we have
  \[
 \int_{\om} e^{\th T_{n,l}(a)} \ud a \leq e^{\vep_2 n}|\om|. 
    \]
\end{Lem}

\begin{proof}
Let $\hat{\om} \subset \om$ be a subset of $\om$ such that every parameter $a \in \hat{\om}$ has $s$ number of free returns into $U$ after escape situations. So $T_{n,l}(a)$ consists of $s$ terms of the form $E_l(a,\nu_j(a))$, $j=1,\ldots, s$, where $\nu_1=\nu$. Recall that $E_l(a,\nu_j(a)) = 0$ if the return is shallow. Set $\xi_{n,l}(\om) = \xi_n(\om)$. Every parameter $a \in \hat{\om}$ has a nested sequence of corresponding intervals so that $a \in \om^s \subset \om^{s-1} \subset \ldots \subset \om^1 \subset \hat{\om}$, such that $\xi_{\nu_{j+1}(a)}(\om^j)$ is in escape position and $\xi_{\nu_j(a)}(\ti{\om}^{j})$ is an essential return, $\om^j \subset \ti{\om}^{j}$. We have $\ti{\om}^{1}=\om$, by assumption. We also see that $E_l(a,\nu_j(a))$ is constant on $\om^j=\om^j(a)$ but not on $\om^{j-1}$. We think of $\om^1 = \om^1(a) \subset \hat{\om}$ as an interval around $a$ which has escaped at time  $\nu_2 = \nu_2(a)$ (possibly earlier). Then $\om^2$ is another smaller interval around $a$ which has escaped at time $\nu_3$ (possibly earlier) and so on. In the construction one should think of $\om$ as contained in some larger interval $\om^0$, $\om \subset \om^0$ where $\xi_{\nu_1}(\om^0)$ is in escape position, and where $\xi_{\nu_1}(\om)$ is an essential return.

Since $T_{n,l}(a) = \sum_{j=1}^{s} E_l(a,\nu_j)$, and $E_l(a,\nu_j(a))$ is constant on $\om^{j}$ but not on $\om^{j-1}$, we get,
  \[
    \int_{\om^{s-1}} e^{\th T_{n,l}(a)} \ud a = \sum_{j=1}^{s-1} e^{\th E_l(a,\nu_j)} \int_{\om^{s-1}} e^{\th E_l(a,\nu_{s})} \ud a.
  \]
  Now, $\xi_{\nu_{s}}(\om^{s-1})$ is in escape position and therefore each interval $\om^{s} \subset \om^{s-1,r}$ where $\diam(\xi_{\nu_{s}}(\om^{s-1,r})) \sim e^{-r}$. Also $\om^{s-1}$ is a union of disjoint intervals $\om^{s-1,r}$, i.e.
  \[
\om^{s-1} = \bigcup_{r=\De}^{\infty} \om^{s-1,r}.
    \]
Recall that the escape time $E_l(a,\nu_{s}(a)) = 0$ for $a \in \om^{s-1,r}$ if $r \leq  2\De$. By Lemma \ref{escape-times} we have
\begin{align}
  \int_{\om^{s-1}} e^{\th E_l(a,\nu_{s})} \ud a &\leq |\om^{s-1}|  + \sum_{r \geq 2\De} \int_{\om^{s-1,r}} e^{\th E_l(a,\nu_{s})} \ud a  \nonumber \\
                                                &\leq |\om^{s-1}| + \sum_{r=2\De}^{\infty} \biggl( \int_{ \{a \in \om^{s-1} : E_l(a,\nu_{s}(a)) \geq 2hr \} } e^{\th E_l(a,\nu_{s})} \ud a \nonumber \\
  &+ \int_{ \{ a \in \om^{s-1} : E_l(a,\nu_{s}(a)) \leq 2hr\} } e^{\th E_l(a,\nu_{s})} \ud a \biggl) \nonumber \\
                                                &\leq |\om^{s-1}| +  C \sum_{r=2\De}^{\infty} (e^{r/3} + e^{-r/3})|\om^{s-1,r}|.
\end{align}
Since $\xi_{\nu_{s}}(\om^{s-1})$ is in escape position, by the Main Distortion Lemma the parameters $a$ that enter into the set where  $\diam(\xi_{\nu_{s}}(\om^{s-1,r})) \sim e^{-r}$ has measure $\sim \frac{e^{-r}}{\de} |\om^{s-1}|$. Therefore,
\begin{align}
  \int_{\om^{s-1}} e^{\th E_l(a,\nu_{s})} \ud a &\leq |\om^{s-1}| +  C \sum_{r=2\De}^{\infty} (e^{r/3} + e^{-r/3})  \frac{e^{-r}}{\de} |\om^{s-1}|  \nonumber \\
  &= |\om^{s-1}| (1 + Ce^{-\De/3}) = |\om^{s-1}|(1+ \eta(\De)),
\end{align}
where $\eta(\De) \raw 0$ as $\De \raw \infty$. 

Next, we want to compute the integral over $\om^{s-2}$: Again $\xi_{\nu_{s-1}}(\om^{s-2})$ is in escape position and therefore $\om^{s-2}$ is subdivisioned into disjoint intervals of the type $\om^{s-2,r} \subset \om^2$ as $\om^{s-1}$:
\[
\om^{s-2} = \bigcup_{r=\De}^{\infty} \om^{s-2,r}.
  \]
  Since $E_l(a,\nu_j(a))$ is constant on $\om^j$, we now compute
  \begin{align}
    \int_{\om^{s-2,r}} e^{\th (E_l(a,\nu_s) + E_l(a,\nu_{s-1}))} \ud a
    &= \sum_{\om^{s-1} \subset \om^{s-2,r}} e^{\th E_l(a,\nu_{s-1})} \int_{\om^{s-2,r} \cap \om^{s-1}} e^{\th E_l(a,\nu_{s})} \ud a \nonumber \\
    &\leq \sum_{\om^{s-1} \subset \om^{s-2,r}} e^{\th E_l(a,\nu_{s-1})} (1 + \eta(\De))|\om^{s-1}|  \nonumber \\
    &= (1+ \eta(\De)) \int_{\om^{s-2,r}} e^{\th E_l(a,\nu_{s-1})} \ud a. 
\end{align}
Thus,
\begin{align}
  \int_{\om^{s-2}} &= \sum_{r \geq 2 \De} \int_{\om^{s-2,r}} e^{\th (E_l(a,\nu_{s-1}) + E_l(a,\nu_s))} \ud a   \nonumber \\
                   &\leq (1 + \eta(\De)) \sum_{r \geq 2 \De} \int_{\om^{s-2,r}} e^{\th E_l(a,\nu_{s-1})} \ud a \nonumber \\
  &\leq (1 + \eta(\De)) \int_{\om^{s-2}} e^{\th E_l(a,\nu_{s-1})} \ud a \leq (1 + \eta(\De))^2 |\om^{s-2}|. 
\end{align}
Repeating this $s$ times and noting that $s \leq n$ trivially and that $\eta(\De) \raw 0$ as $\De \raw \infty$, we get
\[
\int_{\om^0} e^{\th T_{n,l}(a)} \ud a \leq (1 + \eta(\De))^s |\om^0| \leq e^{\vep_2 n} |\om^0|.
\]
Since this holds for every set of the type $\hat{\om}$ (and letting $s$ vary) the lemma follows. 
\end{proof}

Finally we can prove the main goal in this section.

\begin{Lem}  \label{esctimes-estimate}
Let $\tau > 0$ be such that $\tau \th > \vep_2$ and suppose that $\xi_{\nu,l}(\om)$ is a deep essential return with $\dist(\xi_{\nu,l}(\om), Crit_{\om}) \sim_{\sqrt{e}} e^{-r}$, $n \leq \nu \leq 2n$, and $\om \subset \EE_{\nu,l,\star} (\ga) \cap \BB_{\nu,l,\star}$ for some $\ga \geq \ga_I$. Suppose also that all parameters $a \in \om \cap \BB_{2n,l}$ has that $a \in \EE_{2n,l}(\ga_I)$. Then
  \[
m(\{a \in \om : T_n(a) \geq \tau n \} ) \leq e^{n(\vep_2 - \th \tau)} |\om|. 
    \]
  \end{Lem}
  \begin{proof}
    We have by Lemma \ref{escape-measure},
    \begin{equation}
      e^{\th \tau n} m(\{a \in \om : T_n(a) \geq \tau n \} ) \leq \int_{ \{ T_n(a) > \tau n  \} } e^{\th T_n(a)} \ud a \leq \int_{\om} e^{\th T_n(a)} \ud a \leq e^{\vep_2 n} |\om|,
\nonumber
    \end{equation}
    from which we conclude that
    \[
m(\{a \in \om : T_n(a) \geq \tau n \} ) \leq e^{n(\vep_2 - \th \tau)} |\om|.
      \]
    \end{proof}

\section{Conclusion and proof of the main theorem}

We make induction over time intervals of the type $[n,2n]$. By Lemma \ref{startlemma}, for a sufficiently small starting interval $\om_0 = (-\vep,\vep)$ around the starting map $f_0$, there are numbers $N_l$ such that $\xi_{N_l,l}(\om_0)$ has grown to the large scale or returned into $U$ with $\dist(\xi_{N_l,l}(\om_0),Jrit_{\om_0}) / (\log (  \dist(\xi_{N_l,l}(\om_0),Jrit_{\om_0})))^2 \leq \diam(\xi_{N_l,l}(\om_0)) $, (i.e. in the case of a return, it has to be essential). Suppose, without loss of generality, that the first critical point ($l=1$) has that $N_1 = \min(N_l)$. Let $\nu_0 \geq N_1$ be the first return into $U$. It follows that $\xi_{\nu_0,1}(\om_0)$ is an essential return.


If $\nu_0 > 2N_1$ then it means we have no more returns in $[N_1,2N_1]$ for $l=1$ so we go on to the next critical point. To start, put $n=N_1$. For each critical point, we consider the returns $\nu_j \in [n,2n]$ and delete parameters according to the basic approach rate condition. If $\hat{\om}_0 \subset \om_0$ is the set that is left from $\om_0$ when we have deleted parameters not satisfying this condition up until time $2n$, then by Lemma \ref{ba-measure}, 
\[
|\hat{\om_0}| \geq (1-e^{-\al n})|\om_0|.
\]
We make this construction for each critical point, and thereby get a set $\Om_l(2N_1)$, which corresponds to $\hat{\om}_0$ for each $l$, and which contains parameters in $\om_0$ that satisfy the basic assumption for the critical point $c_l$. Up until time $n=N_1$ we see that $\om_0 \subset \EE_{N_1,l}(\ga_B)$, by making $\vep$ sufficiently small. Actually we have a stronger statement at this early stage according to the Starting Lemma (the Lyapunov exponents are close to $\ga_0$), but we do not need that. Moreover, by definition we have $\om_0 \subset \BB_{N_1-1,l}$ for all $l$. Obviously, $\Om_l(2N_1) \subset \BB_{2N_1,l}$. 

If we do not do anything more than keeping the parameters satisfying the basic approach rate condition, the Lyapunov exponent may drop in the time window $[n,2n]$, and over time we may lose too much. 
Every return in this time window has a bound period $p_j \leq \nu_j (2 K\al /\ga_I) = \hat{\al} \nu_j \leq 2n$, for the returns $\nu_j \in [n,2n]$, where we have set $\hat{\al} = 2 K \al/\ga_I$. Hence we can use the expansion of the early orbits up until time $2n$ for all such bound periods. We also note that by Lemma \ref{boundexp}, the bound periods are bounded from below by $Kr_j/(2\Ga) \geq K\De/(2\Ga)$. Let $L_j$ be the corresponding free periods. For every parameter which satisfies the basic approach rate condition, by Lemmas \ref{oel} and \ref{boundexp}, using that $a \in \EE_{n,l}(\ga_B)$, we have, if $\de=e^{-\De}$ is sufficiently small,
\begin{multline} \label{exp-degradation}
  |Df^{2n}(v_l(a))| \geq C_0 e^{\ga_B n} \prod_j \bigl( e^{p_j (\ga/(2K))} C' e^{L_j \ga_H} \bigl)  \\
  \geq e^{\ga_B n} e^{\sum_j p_j (\ga/(4K)) + \ga_H L_j} \geq C_0 e^{\ga_B (1/2) 2n}.
  \end{multline}

Hence up until time $2n=2N_1$, we may have lost some part of the starting Lyapunov exponent ($\ga_B$), but at each return it does not go below $\ga_B/2 > \ga_I$, where $\ga_I$ is a lower bound for most lemmas in the induction process. However, precisely after a return the exponent may drop, but not more than $4 K \al$ because of the basic approach rate assumption (the $2 \al$ is replaced by $4 \al$ to eat up constants), and in general each parameter $a$ we are considering  belongs to $\EE_{n,l}(\ga)$ for some $\ga \geq \ga_B/2 - 4 K \al \geq \ga_I$. 

  Therefore we may have to delete more parameters, that return too often and too deep, in order to restore the Lyapunov exponent for the remaining parameters. This is handled in the section about large deviations. The large deviation argument estimates the set of those parameters that spend a too large portion of the time in $[n,2n]$ reaching escape positions. Since the escape period is set to zero for shallow returns, i.e. for returns into $U \sm U^2$, the orbits $\xi_{n,l}(a)$ outside $U^2$ can be considered as free periods. Using Lemma \ref{oel} for this neighbourhood $U^2$ also gives uniform expansion until the next return (let us use the same exponent $\ga_H > 0$ for those free periods). Since each bound period for a deep return into $U^2$ is contained in an escape period, we now consider those bound periods $\ti{p}_j$ in $[n,2n]$ and the corresponding free periods $\ti{L}_j$ outside $U^2$. If the parameter $a$ is such that $T_{n}(a) \leq \tau n$ where $0 < \tau < 1$ then 
  \begin{equation} \label{exp-restorage}
  |Df^{2n}(v_l(a))| \geq C_0 e^{\ga_B n} e^{\sum_j \ti{p}_j (\ga/(4K))} e^{\ga_H \ti{L}_j} \geq C_0 e^{\ga_B n} e^{(1-\tau) n \ga_H }.
\end{equation}
According to the definition, $\ga_B = (3/4) (1-\tau) \min(\ga_H,\ga_0)$, and hence the Lyapunov exponent is restored: 
\[
|Df^{2n}(v_l(a))| C_0 \geq  e^{\ga_B 2n}.
  \]

Let us now turn to the general case where we use induction. Assume that we have constructed $\Om_l(n)$ for every $l$ and that the sets $\Om_l(n)$ are ``good'' in the following sense. We assume that each partition element $\om \subset \Om_l(n)$ belongs to $\EE_{n,l}(\ga_B) \cap \BB_{n,l}$, i.e. $\Om_l(n) \subset  \EE_{n,l}(\ga_B) \cap \BB_{n,l}$. The sets $\Om_l(n)$ have their own structure and should not be mixed until at the very end, because the partition elements in each such set may differ, and and intersection therefore can destroy these elements. 

  
For simplicity assume that $\nu=n$ is a return time for $l$. By definition of $\EE_{\nu,l}(\ga)$ and $\BB_{\nu,l}$, we can use the binding information for all critical point up until time $(2K \al/\ga_I) \nu  = \hat{\al} \nu$. First let us from $\Om_l(n)$ delete parameters so that we can use the binding information of all other critical points $j \neq l$  for a longer time, in the next time window $[2n,4n]$, i.e. we want to consider $\EE_{\nu,l,\star}(\ga) \cap \BB_{\nu,l,\star}$. The point is now that the partition elements we are deleting, i.e. parameters belonging to $(\EE_{\nu,l,\star}(\ga) \cap \BB_{\nu,l,\star}) \sm (\EE_{\nu,l}(\ga) \cap \BB_{\nu,l})$,  by this procedure are much larger than the partition elements in $\Om_l(n)$ (this was originally observed by M. Benedicks). Indeed, if we look at the length of  $\xi_{m,j}(\om_1)$ where $\om_1$ is a partition element that got deleted at some time (return) $m \leq 2 \hat{\al} n$ then, by Lemma \ref{transv1}, 
\[
\diam(\xi_{m,j}(\om_1))\sim |\om_1||Df^m(v_j(a))| \leq |\om_1|e^{\Ga m}.
\]
By the basic assumption, and since $\om_1$ got deleted at time $m$, we have
\[
  \diam(\xi_{m,j}(\om_1)) \sim \dist(\xi_{m,j}(\om_1),Crit_{\om_1}) / (\log (\dist (\xi_{m,j}(\om_1),Crit_{\om_1})))^2 \geq e^{-3 \al m},
\]
so
\[
e^{-3\al m} \leq C \diam(\xi_{m,j}(\om_1)) \leq C |\om_1|e^{\Ga m}.
\]

On the other hand, the partition elements at time $n$ or higher, are much smaller. This can be seen as follows. Let $\om_2$ be a partition element at time $n$. Since $\diam(\xi_{n,j}(\om_2)) \leq S$, we have
\[
S \geq \diam(\xi_{n,j}(\om_2)) \sim |\om_2| |Df^n(v_j(a))| \geq |\om_2| C_0 e^{\ga n}.
\]
Therefore, since $m \leq 2 \hat{\al} n$,
\[
\frac{|\om_1|}{|\om_2|} \geq C \frac{e^{-(3\al + \Ga)m}}{e^{-\ga n}} \geq C e^{ (\ga - 2 \hat{\al} (3 \al + \Ga) )n} \gg 1.
\]
Hence $\om_2$ is much smaller than $\om_1$. This means that when deleting partition elements in $\Om_l(n)$ that do not satisfy the basic approach rate condition until time $2\hat{\al}n$ for other critical points $j \neq l$, in the time window $[\hat{\al}n, 2\hat{\al}n]$, we do not destroy the partition elements; we only delete whole partition elements of the type $\om_2 \in \Om_l(n)$ that intersect partition elements of the type $\om_1$ that was deleted at the time scale $\sim \hat{\al}n$. 

Starting from the partition elements in $\Om_l(n) \subset \EE_{n,l}(\ga) \cap \BB_{n,l} $ and passing to $ \Om_l(n,\star) \subset \EE_{\nu,l,\star}(\ga) \cap \BB_{\nu,l,\star}$ is therefore harmless and the measure deleted is
\begin{equation} \label{star-measure}
|\Om_l(n,\star)| \geq (1-C e^{-\al \hat{\al}n}) |\Om_l(n)|. 
\end{equation}
We have now constructed $\Om_l(n,\star)$ and want to pass to $\Om_l(2n) \subset \EE_{2n,l}(\ga) \cap \BB_{2n,l}$. 
Passing from $\Om_l(n,\star)$ to $\Om_l(2n)$, we have to delete parameters that do not satisfy the basic approach rate condition for critical point $c_l$ and also delete those parameters that have too many too deep returns in $[n,2n]$. We have seen by equation (\ref{exp-degradation}), that the Lyapunov exponent can decrease to $(1/2) \ga_B > \ga_I$ during the period form $n$ to $2n$. We also have to take into account the blind escapes, see Remark \ref{blind-escapes}, which constitute a small portion, $\leq Ce^{-qn}$, of the original set of parameters. For those parameters whose escape periods transcend into $[2n,4n]$ (these are the escape periods for the last return $\nu_s$ discussed in Remark \ref{blind-escapes}), the Lyapunov exponent may drop slightly below $\ga_B/2$, but never below $\ga_I$ (if $\al$ is sufficiently small, see the condition on p. 4). By Lemmas \ref{ba-measure} and \ref{escape-measure} we get the estimate
\[
  |\Om_l(2n)| \geq (1-e^{-\al n})  (1- Ce^{-(\th \tau - \vep_2)n} )(1-Ce^{-qn}) |\Om_l(n,\star)|. 
\]

Together with (\ref{star-measure}), we get, for some $\be > 0$, 
\begin{align}
  |\Om_l(2n)|  &\geq  (1-e^{-\al n})  (1- Ce^{-(\th \tau - \vep_2)n})(1-Ce^{-qn}) (1-C e^{-\al \hat{\al}n})    |\Om_l(n)| \nonumber \\
  &\geq  (1-e^{-\be n}) |\Om_l(n)|. \nonumber
  \end{align}
  It follows that $\Om_l(2n) \subset \EE_{2n,l}(\ga) \cap \BB_{2n,l}$, where $\ga \geq \ga_B$ by the choice of $\ga_B$
  (possibly, if $n$ is just after a return time, $\ga \geq \ga_B - 4 K \al$). We are then back to the same situation at time $2n$ as we were for time $n$ and the induction argument goes on forever.  

Let $M \geq 2$. Choosing the constants correctly, in this way we construct, for each critical point, a set $\Om_l(n) \subset \EE_{n,l}(\ga_B- 4 K \al) \cap \BB_{n,l}$ with measure at least $(1-1/(2Md)) |\om_0|$, that holds for $n > 0$, where $d$ is the degree of $f$. Passing to the limit, as $n \raw \infty$, we get that the measure of parameters that satisfies the CE-condition for all $n > 0$ is estimated by
    \[
\lim\limits_{n \raw \infty} m \biggl( \bigcap_l \Om_l(n) \biggr) \geq  (1-\frac{1}{M}) |\om_0|. 
      \]
Since $M$ can be chosen arbitrarily large, it follows that $f_0$ is a Lebesgue density point of CE-maps. 

\newpage

\bibliographystyle{plain}
\bibliography{ref}

\end{document}